\documentclass[a4paper]{article}

\usepackage{a4wide}
\usepackage{amsmath,amssymb,amsthm,bbm}
\usepackage[american]{babel}
\usepackage[dvips]{graphicx}

\graphicspath{{./}{figures/}}

\newcommand{\ds}{\displaystyle}
\newcommand{\bs}{\boldsymbol}

\newtheorem{theorem}{Theorem}
\newtheorem*{theorem_preview}{Theorem}
\newtheorem{lemma}[theorem]{Lemma}
\newtheorem{proposition}[theorem]{Proposition}
\newtheorem{corollary}[theorem]{Corollary}
\newtheorem{definition}{Definition}
\theoremstyle{definition}\newtheorem{assumption}{Assumption}
\theoremstyle{remark}\newtheorem*{remark}{Remark}

\newcommand{\R}{\mathbb{R}} 
\newcommand{\Rd}{\R^d} 
\newcommand{\Rplus}{[0,\,+\infty)} 
\newcommand{\Z}{\mathbb{Z}} 
\newcommand{\Zd}{\Z^d} 
\newcommand{\N}{\mathbb{N}} 
\DeclareMathOperator*{\bigcart}{\raisebox{-0.9ex}{\text{\Huge $\times$}}} 

\newcommand{\funct}[3]{#1:#2\to #3} 
\newcommand{\ind}[1]{\mathbbm{1}_{#1}} 
\newcommand{\smoothind}[1]{\chi_{#1}} 
\newcommand{\C}[2]{C(#1;\,#2)} 
\newcommand{\Cc}{C^\infty_c(\Rd)} 
\newcommand{\lip}[1]{\operatorname{Lip}(#1)} 
\newcommand{\Lip}{\textup{Lip}_1(\Rd)} 
\newcommand{\supp}[1]{\operatorname{supp}{#1}} 
\newcommand{\CTPone}{\C{[0,\,T]}{\Pone}} 

\newcommand{\abs}[1]{\left\vert #1\right\vert} 
\newcommand{\infnorm}[1]{\Vert #1\Vert_\infty} 

\newcommand{\lint}{\int\limits} 
\newcommand{\intg}[3]{\lint_{#1}#2\,d#3} 
\newcommand{\intRd}[2]{\intg{\Rd}{#1}{#2}} 

\newcommand{\B}{\mathcal{B}(\Rd)} 
\newcommand{\leb}{\mathcal{L}} 
\newcommand{\PP}{\mathcal{P}} 
\newcommand{\Pone}{\PP_1(\Rd)} 
\newcommand{\Ptwo}{\PP_2(\Rd)} 
\newcommand{\wass}[2]{W_1(#1,\,#2)} 
\renewcommand{\P}{P} 
\newcommand{\Exp}[1]{\mathbb{E}[#1]} 
\newcommand{\condexp}[1]{\hat{#1}} 

\newcommand{\Dt}{\Delta{t}} 
\newcommand{\appr}[1]{\tilde{#1}} 

\newcommand{\vd}{v_\textup{d}} 
\newcommand{\Vd}{V_\textup{d}} 
\newcommand{\neighb}{U} 
\newcommand{\Fmax}{C_F} 
\newcommand{\rot}{\mathcal{R}} 
\newcommand{\rotang}{\vartheta} 
\newcommand{\ivect}{\mathbf{i}} 
\newcommand{\kvect}{\mathbf{k}} 

\newcommand{\ie}{i.e.} 
\newcommand{\eg}{e.g.} 
\newcommand{\wrt}{w.r.t.\ } 
\newcommand{\eq}{Eq.~} 
\newcommand{\ccdot}{\bullet} 
\newcommand{\incond}[1]{\bar{#1}} 
\newcommand{\ball}[2]{B_{#1}(#2)} 
\newcommand{\eps}{\varepsilon} 
\newcommand{\diam}[1]{\operatorname{diam}{#1}}
\newcommand{\xlag}{\bar{x}}

\usepackage{color}

\graphicspath{{./}{figures/}}

\title{Existence and approximation of probability measure solutions to models of collective behaviors}
\author{Andrea Tosin\thanks{Corresponding author. Tel: (+39) 011.090.7531. {\it E-mail address}: 
        	\texttt{andrea.tosin@polito.it}.},
			Paolo Frasca\\[0.5cm]
			{\small\it\ Department of Mathematics, Politecnico di Torino}\\[-0.1cm]
			{\small\it Corso Duca degli Abruzzi 24, 10129, Torino, Italy}
	    }
\date{}
	    
\begin{document}
\maketitle

\begin{abstract}
In this paper we consider first order differential models of collective behaviors of groups of agents based on the mass conservation equation. Models are formulated taking the spatial distribution of the agents as the main unknown, expressed in terms of a probability measure evolving in time. We develop an existence and approximation theory of the solutions to such models and we show that some recently proposed models of crowd and swarm dynamics fit our theoretic paradigm.
\end{abstract}

\section{Introduction}
This paper deals with mathematical models of collective behaviors of groups of interacting agents, such as human crowds and swarms. The reference framework is that of first order differential models ruled by the principle of conservation of mass and supplemented by a kinematic description of the behavioral strategy developed by the agents. The number of individuals is assumed to be finite, though arbitrarily large, their state being represented by their position evolving in time in the Euclidean space $\Rd$, $d\geq 1$. Rather than looking at the agents singularly, we abstract their spatiotemporal evolution into that of a suitable probability measure, representing the law of their positions understood as random variables. This allows us to provide a unified theory for both discrete and continuous models.

There are three main contributions of this paper. (i) We outline a basic set of modeling assumptions, which allow us to prove the existence of probability measure solutions to a broad class of models of collective behaviors of the kind described above. (ii) Under the very same assumptions, complemented with a suitable condition on the time and space discretization, we provide a convergence result of an \emph{ad hoc} numerical scheme, originally proposed in~\cite{piccoli2010tem}, for approximating the solutions to such models. (iii) We reinterpret the \emph{rendez-vous}, swarm, and crowd dynamics models developed in~\cite{canuto2008eaa,cristiani2010eai,piccoli2009pfb,piccoli2010tem} in the light of our probabilistic description, showing that they comply with the above modeling assumptions and are therefore in the scope of our existence and approximation theory.

In more detail, the paper is organized as follows. After this Introduction, Section~\ref{sec:notations} briefly introduces and explains the main notations and notions used throughout the other sections. Section~\ref{sec:prob.stat} proposes a probabilistic interpretation of the dynamics of systems of interacting agents, discussing both the indefinite mass conservation equation and the related Cauchy problems. Then it presents the modeling assumptions and offers an overview of the results proved in the subsequent sections. Section~\ref{sec:existence} deals with the existence of solutions and Section~\ref{sec:approx} with the convergence of the approximation scheme. These two sections may be skipped by readers not interested in the technical details of the proofs. Section~\ref{sec:models} addresses the above-cited crowd and swarm dynamics models, showing that they fit the theory in all cases interesting for the applications. Finally, Section~\ref{sec:case.study} gives an example of how ODE-based discrete models can be explicitly recovered from our measure-theoretic framework. In addition, it shows by means of numerical tests the convergence of the computational scheme discussed in Section~\ref{sec:approx} to the ODE solutions of such models.

\subsection{Notations and background} \label{sec:notations}
In this section we quickly review the main notations and notions that we will extensively use in the paper.

\paragraph{Functions and function spaces.} We denote by $\C{A}{B}$ the space of continuous functions $\funct{f}{A}{B}$. The set $B$ is usually omitted if it is $\R$. Coherently, we denote by $\Cc$ the space of real-valued infinitely differentiable functions with compact support in $\Rd$.

The indicator function of a set $A$ is $\ind{A}$, namely $\ind{A}(x)=1$ if $x\in A$ and $\ind{A}(x)=0$ if $x\in A^c$.

\paragraph{Norms.} We use $\abs{\cdot}$ for the Euclidean norm in $\Rd$ and $\cdot$ for the corresponding inner product. The open Euclidean ball with center $x\in\Rd$ and radius $R\ge 0$ is denoted by $\ball{R}{x}$. If $\funct{f}{\Rd}{\Rd}$ is bounded, its $\infty$-norm is $\infnorm{f}:=\sup_{x\in\Rd}\abs{f(x)}$.

\paragraph{Measures.} 
Let $\B$ be the Borel $\sigma$-algebra on $\Rd$. If $\mu$ is a measure on the measurable space $(\Rd,\,\B)$ and $\funct{f}{\Rd}{\Rd}$ is Borel, the integral of $f$ \wrt $\mu$ over a measurable set $A$ is denoted by $\int_A f\,d\mu$, or by $\int_A f(x)\,d\mu(x)$ when it is necessary to emphasize the variable of integration. The Lebesgue measure in $\R^d$ is denoted $\leb^d$. However, for integrals with respect to $\leb^d$ we will prefer the usual symbol $dx$ to $d\leb^d(x)$. If $f$ is Borel, we denote by $f\#\mu$ the push forward of $\mu$ through $f$. Specifically, $f\#\mu$ is the measure defined by the relation
\begin{equation}
	\intRd{\eta}{(f\#\mu)}=\intRd{\eta\circ f}{\mu}
	\label{eq:def.pushfwd}
\end{equation}
for all bounded (or $f\#\mu$-integrable) and Borel $\funct{\eta}{\Rd}{\R}$. Taking $\eta=\ind{A}$, $A\in\B$, gives in particular $(f\#\mu)(A)=\mu(f^{-1}(A))$.

\paragraph{Probability spaces.}
We denote by $\Pone$ the space of probability measures on $(\Rd,\,\B)$ whose first moment is finite, \ie, $\int_{\Rd}|x|\,d\mu(x)<+\infty$. The space $\Ptwo$, that we will also occasionally mention, is defined analogously using the second moment. Given two probability measures $\mu,\,\nu\in\Pone$, their Wasserstein distance is defined to be
\begin{equation*}
	\wass{\mu}{\nu}=\sup_{\varphi\in\Lip}\intRd{\varphi}{(\nu-\mu)},
\end{equation*}
where $\Lip$ is the space of Lipschitz continuous functions $\funct{\varphi}{\Rd}{\R}$ with Lipschitz constant $\lip{\varphi}\leq 1$. It can be shown that $W_1$ is a metric on $\Pone$ and that $(\Pone,\,W_1)$ is complete (see \eg, \cite[Proposition 7.1.5]{ambrosio2008gfm}).

Finally, to deal with curves in $\Pone$ parameterized by time, $[0,\,T]\ni t\mapsto\mu_t\in\Pone$, we introduce the space $\CTPone$, which is complete with the metric
\begin{equation*}
	\operatorname{dist}(\mu_\ccdot,\,\nu_\ccdot):=\sup_{t\in[0,\,T]}\wass{\mu_t}{\nu_t}.
\end{equation*}

\section{Problem statement and main results} \label{sec:prob.stat}
In this section we present our approach to the modeling of systems of interacting agents by means of probability measures, and we give an overview of our results.

\subsection{Probabilistic description of systems of interacting agents}
Crowds and swarms can be thought of, in abstract, as systems of $N$ interacting agents in the physical space $\Rd$. The evolution of such systems in time is described by tracing the agent positions at successive instants.
Assume that the position of the $i$-th agent at time $t$ is a random variable $X^i_t$ from a fixed (\ie, time-independent) abstract probability space $(\Omega,\,\mathcal{F},\,\P)$ to the measurable space $(\Rd,\,\B)$. The probability $\P$ is naturally transported by $X^i_t$ onto the new probability $\mu_t:=X^i_t\#\P$ on $(\Rd,\,\B)$, called the law of $X^i_t$. The fact that $\mu_t$ does not depend on the agent label $i$ means that agents are \emph{indistinguishable}: the probability of finding a certain agent somewhere in $\Rd$ at time $t$ is the same for all agents, namely, given $A\in\B$, $\P(\{X^i_t\in A\})=\mu_t(A)$ for all $i=1,\,\dots,\,N$.

Let us now fix a measurable set $A\subseteq\Rd$ and count the average number of agents contained in $A$ at time $t$. This amounts to introducing the new random variable $Y_{t,A}:\Omega\to\N$ defined as
\begin{equation*}
	Y_{t,A}=\sum_{i=1}^N\ind{\{X^i_t\in A\}}
\end{equation*}
and taking its expectation, that we can compute as follows:
\begin{equation*}
	\Exp{Y_{t,A}}=\sum_{i=1}^N\P(\{X^i_t\in A\})=N\mu_t(A).
\end{equation*}
Notice that $\Exp{Y_{t,\ccdot}}$, thought of as a map on $\B$, is a finite positive measure, say $m_t$, over the measurable space $(\Rd,\,\B)$, such that $m_t(\Rd)=N$ all $t$. It is straightforward to identify $m_t$ with the \emph{mass} of the system at time $t$. From the above calculation we see that $m_t$ is proportional to the probability measure of the distribution of the agents: $m_t(A)=N\mu_t(A)$, where $m_t(A)$ is the mass of the set $A\in\B$ at time $t$. In addition, $m_t(\Rd)=N$ for all $t$, \ie, the total mass of the system is constant in time.

The latter observation suggests that we can assume the principle of conservation of the mass, stating that the mass of any measurable set $A$ may change in time only because of inflow or outflow of mass from the boundary $\partial A$. In other words, the mass is neither created nor destroyed but only moved across the domain. This is expressed by postulating the \emph{continuity equation} (or \emph{mass conservation equation}) for the evolution of the measure $m_t$:
\begin{equation}
	\frac{\partial m_t}{\partial t}+\nabla\cdot(m_t v_t)=0,
	\label{eq:mass}
\end{equation}
where $v_t(x)$ is the velocity at time $t$ in the point $x\in\Rd$. In systems of interacting agents the velocity is likely to be affected by the distribution of the agents themselves. Due to the proportionality between $m_t$ and $\mu_t$, this implies that $v_t$ may ultimately depend on the probability $\mu_t$. We assume in particular that $v_t$ depends on $t$ only through $\mu_t$ itself, \ie, that the system is autonomous. Finally, we write $v_t=v[\mu_t]$ to emphasize such a structure of the velocity and notice that \eq\eqref{eq:mass} can be converted into an evolution equation for the probability $\mu_t$:
\begin{equation}
	\frac{\partial\mu_t}{\partial t}+\nabla\cdot(\mu_t v[\mu_t])=0
	\label{eq:cont.eq}
\end{equation}
for $x\in\Rd$ and $t\in(0,\,T]$, where $T>0$ is the final time. If $\mu_t$ solves \eq\eqref{eq:cont.eq}, then $m_t$ formally solves \eq\eqref{eq:mass} with $v_t=v[m_t/N]$.

\subsection{Cauchy problems} \label{sec:cauchy}
By supplementing \eq\eqref{eq:cont.eq} with an initial condition $\incond{\mu}$, the following Cauchy problem is obtained:
\begin{equation}
	\begin{cases}
		\dfrac{\partial\mu_t}{\partial t}+\nabla\cdot(\mu_t v[\mu_t])=0 & \text{in\ } \Rd\times (0,\,T] \\
		\mu_0=\incond{\mu} & \text{in\ } \Rd,
	\end{cases}
	\label{eq:cauchy}
\end{equation}
which models the spatiotemporal evolution of the agent distribution starting from the initial configuration described by $\incond{\mu}$. Derivatives in problem \eqref{eq:cauchy} are meant in the sense of distributions, which leads us to consider the following notion of weak solution:

\begin{definition}[Weak solutions] \label{def:weak.sol}
Given $\incond{\mu}\in\Pone$, we say that $\mu_\ccdot\in\CTPone$ is a weak solution to problem \eqref{eq:cauchy} if
\begin{equation}
	\intRd{\phi}{\mu_t}=\intRd{\phi}{\incond{\mu}}+\lint_0^t\intRd{v[\mu_\tau]\cdot\nabla\phi}{\mu_\tau}\,d\tau,
	\quad \forall\,\phi\in\Cc,\ \forall\,t\in[0,\,T].
	\label{eq:weak.sol}
\end{equation}
\end{definition}

As far as the probabilistic interpretation is concerned, we take the initial condition into account by understanding $\mu_t$ in \eqref{eq:cauchy} as the law of the random variable $\condexp{X}^i_t:=\Exp{X^i_t\vert X^i_0}$, the expectation of $X^i_t$ conditioned to its initial value $X^i_0$. In practice, $\mu_t$ is reinterpreted as the distribution of the position of the $i$-th agent subject to the distribution of the corresponding initial position. From the theory of conditional expectation, we known that $\condexp{X}^i_t$ is a function of $X^i_0$, \ie, there exists a Borel mapping $\funct{\gamma_t}{\Rd}{\Rd}$ such that $\condexp{X}^i_t=\gamma_t(X^i_0)$. This implies, in particular,
\begin{equation}
	\mu_t=\gamma_t\#\incond{\mu}.
	\label{eq:rep.formula}
\end{equation}

In order for \eq\eqref{eq:rep.formula} to result in a representation formula for the solutions of problem \eqref{eq:cauchy}, a more precise characterization of the function $\gamma_t$ is needed. Formally, we plug \eq\eqref{eq:rep.formula} into \eq\eqref{eq:weak.sol} and compute
\begin{equation*}
	\intRd{[\phi(\gamma_t(x))-\phi(x)]}{\incond{\mu}(x)}=
		\lint_0^t\intRd{v[\gamma_\tau\#\incond{\mu}](\gamma_\tau(x))\cdot
			\nabla{\phi}(\gamma_\tau(x))}{\incond{\mu}(x)}\,d\tau
\end{equation*}
for an arbitrarily fixed test function $\phi$. Next we notice that the integrand at the right-hand side can be read as the derivative \wrt $\tau$ of the function $\phi(\gamma_\tau(x))$, provided we identify $\frac{\partial}{\partial t}\gamma_t(x)$ with $v[\gamma_t\#\incond{\mu}](\gamma_t(x))$. Under this assumption we get
\begin{align*}
	\intRd{[\phi(\gamma_t(x))-\phi(x)]}{\incond{\mu}(x)} &=
		\lint_0^t\intRd{\frac{\partial}{\partial\tau}\phi(\gamma_\tau(x))}{\incond{\mu}(x)}\,d\tau
\intertext{and further, interchanging the order of integration at the right-hand side,}
	&=\intRd{[\phi(\gamma_t(x))-\phi(\gamma_0(x))]}{\incond{\mu}(x)}.
\end{align*}
With the additional condition $\gamma_0(x)=x$ (\ie, $\gamma_0$ is the identity function in $\Rd$), this shows that $\mu_t$ represented by \eq\eqref{eq:rep.formula} is formally a weak solution to the Cauchy problem \eqref{eq:cauchy}.

To sum up, $\gamma_t$ has been characterized as a function such that
\begin{equation}
	\begin{cases}
		\dfrac{\partial\gamma_t(x)}{\partial t}=v[\gamma_t\#\incond{\mu}](\gamma_t(x)),
			\quad t\in(0,\,T] \\[3mm]
		\gamma_0(x)=x
	\end{cases}
	\label{eq:gammat}
\end{equation}
for every $x\in\Rd$. In transport theory, such a function is called a \emph{flow map}. The physical interpretation is that $\gamma_t(A)$ is the configuration assumed by the set $A$ at time $t>0$ when transported by the velocity field $v$. Alternatively, for $x\in\Rd$ the mapping $t\mapsto\gamma_t(x)$ is the trajectory of system \eqref{eq:gammat} issuing from $x$.

The method of representing solutions to the Cauchy problem \eqref{eq:cauchy} via flow maps is called \emph{method of the characteristics}. To develop our theory we will mostly prefer a different approach, more suited to treat, by common ideas, existence and approximation of solutions to the models we are interested in. The reader interested in the method of the characteristics is referred to \cite{canizo2011wpt} and references therein for further details.

\subsection{Basic assumptions and results overview}
Models based on \eq\eqref{eq:cont.eq} require to specify the velocity, namely its dependence on the probability $\mu_t$ and on the space variable $x$. Rather than considering a specific model, we outline here a small set of assumptions, which can be possibly regarded as modeling guidelines, whence the whole theory will follow.

\begin{assumption}[Properties of $v$] \label{ass:prop-v}
We assume that the velocity field $v=v[\mu](x)$ satisfies the following properties.
\begin{enumerate}
\item \label{item:ass-bound} {\it Uniform boundedness}: there exists $V>0$ such that
\begin{equation*}
	\abs{v[\mu](x)}\leq V, \quad \forall\,x\in\Rd,\ \forall\,\mu\in\Pone.
\end{equation*}

\item \label{item:ass-lip} {\it Lipschitz continuity}: there exists a constant $\lip{v}>0$, independent of both the space variable and the probability, such that
\begin{equation*}
	\abs{v[\nu](y)-v[\mu](x)}\leq\lip{v}(\abs{y-x}+\wass{\mu}{\nu}),
		\quad \forall\,x,\,y\in\Rd,\ \forall\,\mu,\,\nu\in\Pone.
\end{equation*}

\item \label{item:ass-lin} {\it Linearity \wrt the measure for convex combinations}:
\begin{equation*}
	v[\alpha\mu+(1-\alpha)\nu]=\alpha v[\mu]+(1-\alpha)v[\nu],
		\quad \forall\,\mu,\,\nu\in\Pone,\ \forall\,\alpha\in[0,\,1].
\end{equation*}
\end{enumerate}
\end{assumption}

\bigskip

From this basic set of hypotheses we will be able to prove, in Section~\ref{sec:existence}, existence of solutions to the Cauchy problem \eqref{eq:cauchy} in the appropriate weak sense of Definition~\ref{def:weak.sol}. More precisely, to this end we need the further technical assumption that the initial condition have finite first and second order moments:

\begin{assumption}[Initial condition] \label{ass:in.cond}
We assume that $\incond{\mu}\in\Pone\cap\Ptwo$.
\end{assumption}

Then, our main result in Section~\ref{sec:existence} reads:
\begin{theorem_preview}[cf. Theorem~\ref{theo:existence}]
Under Assumptions~\ref{ass:prop-v},~\ref{ass:in.cond} there exists a weak solution to problem \eqref{eq:cauchy}.
\end{theorem_preview}

Notice that Assumption~\ref{ass:in.cond} is readily satisfied if, for instance, $\incond{\mu}$ has compact support. Indeed, in such a case $\supp{\incond{\mu}}$ is bounded, \ie, there exists a ball $\ball{R}{0}$ of sufficiently large radius $R>0$ such that, for every $p\geq 0$,
\begin{equation*}
	\intRd{\abs{x}^p}{\incond{\mu}(x)}=\intg{\supp{\incond{\mu}}}{\abs{x}^p}{\incond{\mu}(x)}\leq
		\intg{\ball{R}{0}}{\abs{x}^p}{\incond{\mu}}\leq R^p.
\end{equation*}

The compactness of the support of $\incond{\mu}$ makes perfectly sense from the modeling point of view, in fact a crowd or a swarm spread on the whole space would sound quite unrealistic. Assumption~\ref{ass:in.cond} is therefore not restrictive for our purposes.

\bigskip

In Section~\ref{sec:approx} we turn our attention to the approximation of solutions to problem \eqref{eq:cauchy}. We introduce a sequence of grids in $\Rd\times[0,\,T]$ with mesh parameters $h_k$ in space and $\Dt_k$ in time. The index $k$ relates to the grid refinement, in such a way that $h_k,\,\Dt_k\to 0$ when $k\to\infty$. Specifically, we consider the numerical scheme proposed in \cite{piccoli2010tem}, which at each time step seeks an approximation of $\mu_t$ via a probability measure absolutely continuous \wrt to Lebesgue and piecewise constant in space. By introducing a linear-in-time interpolation of such approximate solutions and passing to the limit $k\to\infty$, we obtain the following convergence result:

\begin{theorem_preview}[cf. Theorem~\ref{theo:convergence}]
Under Assumptions~\ref{ass:prop-v},~\ref{ass:in.cond}, suppose that $h_k=o(\Dt_k)$ for $k\to\infty$. If the sequence of approximate solutions converges to some $\mu_\ccdot\in\CTPone$ when the grid is refined then $\mu_\ccdot$ is a weak solution to problem \eqref{eq:cauchy}.
\end{theorem_preview}

Note that convergence when the grid is refined is an assumption of this theorem. In this respect, this result resembles the Lax-Wendroff's Theorem for the numerical approximation of hyperbolic conservation laws (see \eg, \cite{leveque1992nmc}). However, we anticipate that if there exists a bounded subset of $\Rd$, that at each time step and for all level of refinement of the grid contains the supports of all the approximate solutions, then the sequence does converge in $\CTPone$ to some limit, which is then a weak solution to problem \eqref{eq:cauchy} (cf. Corollary~\ref{cor:convergence}).

\bigskip

In Section~\ref{sec:models} we apply the above theory to the models of swarm and crowd dynamics presented in \cite{cristiani2010eai,piccoli2009pfb}. As shown in \cite{cristiani2009mso}, these models can be derived from the common framework provided by \eq\eqref{eq:cont.eq}, with a velocity field of the form
\begin{equation*}
	v[\mu_t](x)=\vd(x)+N\intRd{f(\abs{y-x})r(y-x)\smoothind{\neighb_x}(y)}{\mu_t(y)},
\end{equation*}
where the integral expresses the interactions among the agents. Assuming some minimal regularity of the functions $\vd$, $f$, $r$, $\smoothind{\neighb_x}$, we prove that this velocity complies with Assumption~\ref{ass:prop-v} in case of both isotropic and anisotropic interactions, cf. Sections~\ref{sec:isotropic} and~\ref{sec:anisotropic}, respectively. As a consequence, for the above-mentioned models we deduce existence of probability measure solutions, that can be duly approximated via the numerical scheme discussed in Section~\ref{sec:approx}.

\bigskip

Our results are exemplified in Section~\ref{sec:case.study}, where we show that, given a purely atomic initial measure $\incond{\mu}$, \ie,
\begin{equation*}
	\incond{\mu}=\frac{1}{N}\sum_{l=1}^{N}\delta_{x^l_0} \qquad (x_0^l\in\Rd),
\end{equation*}
a solution to problem \eqref{eq:cauchy} can be found by solving a system of ODEs whose unknowns are the trajectories of the agents. In addition, using a numerical solution of these ODEs as a benchmark, we are able to visualize the convergence of the numerical scheme presented in Section~\ref{sec:approx}.

\section{Existence of solutions} \label{sec:existence}
This section is devoted to give a constructive proof of the existence of solutions to the Cauchy problem \eqref{eq:cauchy}. Under Assumptions~\ref{ass:prop-v},~\ref{ass:in.cond}, the solution is constructed as the limit of a suitable sequence of curves in $\Pone$ parameterized by time $t$.

Let $(\Dt_k)_{k\geq 0}$ be a sequence of time steps such that $\Dt_k\to 0$ when $k\to\infty$. We consider the measures $(\mu_n^k)_{n\geq 0}$ generated recursively as
\begin{equation}
	\begin{cases}
		\mu_{n+1}^k=\gamma_n^k\#\mu_n^k, & \quad n=0,\,1,\,\dots,\,N_k-1, \\
		\mu_0^k=\incond{\mu},
	\end{cases}
	\label{eq:pushfwd}
\end{equation}
where $N_k\in\N$ is such that $N_k\Dt_k=T$ and $\gamma_n^k$ is the (one-step) flow map
\begin{equation}
	\gamma_n^k(x)=x+v[\mu_n^k](x)\Dt_k.
	\label{eq:flowmap}
\end{equation}
It can be shown (see \eg, \cite{piccoli2010tem}) that \eqref{eq:pushfwd} is the explicit time discretization of \eqref{eq:cauchy} at the time instants $t_n^k=n\Dt_k$. In particular, the mapping $\gamma_n^k$ results from the explicit time discretization of problem \eqref{eq:gammat}. Since $\incond{\mu}$ is a probability measure, by induction it is immediate to check that so are all of the $\mu_n^k$'s.

By linear interpolation in time, we define the following curves:
\begin{equation*}
	M_t^k=\sum_{n=0}^{N_k-1}\left[\left(1-\frac{t-t_n^k}{\Dt_k}\right)\mu_n^k+
		\frac{t-t_n^k}{\Dt_k}\mu_{n+1}^k\right]\ind{[t_n^k,\,t_{n+1}^k]}(t).
\end{equation*}
Obviously, $M_t^k$ is a probability measure for each $t\in[0,\,T]$ and each $k\geq 0$.

We will use the curves $M_\ccdot^k$ to construct, in the limit $k\to\infty$, a weak solution to problem \eqref{eq:cauchy}. The proof is divided in two parts, each of which proceeds through a series of technical intermediate steps developed in the next two sections. First, we show that the measures $M_\ccdot^k$ converge to a limit; later, we show that such limit satisfies \eq\eqref{eq:weak.sol}.

\subsection{Convergence of the measures $\bs{M_t^k}$}
In this section we prove that, when $k\to\infty$ and up to subsequences, $M_\ccdot^k$ converges to a limit $\mu_\ccdot$ in $\CTPone$.

We begin by establishing the necessary regularity properties of the iterates $\mu_n^k$ and the curves $M_\ccdot^k$.
\begin{lemma} \label{lemma:mu_n_k}
We have $\mu_n^k\in\Pone\cap\Ptwo$ for all $n=1,\,\dots,\,N_k$ and all $k\geq 0$. In addition:
\begin{equation*}
	\sup_{k\geq 0}{\sup_{1\leq n\leq N_k}\intRd{\abs{x}^p}{\mu_n^k(x)}}<+\infty, \qquad p=1,\,2,
\end{equation*}
\ie, first and second moments of the $\mu_n^k$'s are uniformly bounded.
\end{lemma}
\begin{proof}
\begin{enumerate}
\item \label{firstmom} Let us begin by considering the case $p=1$. 
Note that \begin{align*}
	\intRd{\abs{x}}{\mu_{n+1}^k(x)}&=\intRd{\abs{\gamma_n^k(x)}}{\mu_n^k(x)} \\
	&\leq \intRd{\abs{x}}{\mu_n^k(x)}+\Dt_k\intRd{\abs{v[\mu_n^k](x)}}{\mu_n^k(x)} \\
	&\leq \intRd{\abs{x}}{\mu_n^k(x)}+V\Dt_k,
\end{align*}
then
\begin{equation*}
	\intRd{\abs{x}}{\mu_{n+1}^k(x)}-\intRd{\abs{x}}{\mu_n^k(x)}\leq V\Dt_k,
\end{equation*}
and summing telescopically over $n$ we get
\begin{equation}
	\intRd{\abs{x}}{\mu_n^k(x)}\leq\intRd{\abs{x}}{\incond{\mu}(x)}+VT.
	\label{bound_firstmom}
\end{equation}
Since $\bar\mu\in\Pone$, this implies $\mu_n^k\in\Pone$ for all $n=1,\,\dots,\,N_k$ and all $k\geq 0$. But the right-hand side of \eqref{bound_firstmom} is independent of both $n$ and $k$, hence this also provides the uniform bound for $p=1$.

\item We argue analogously for the case $p=2$. 
We note that 
\begin{equation*}
	\intRd{\abs{x}^2}{\mu_{n+1}^k(x)}=\intRd{\abs{\gamma_n^k(x)}^2}{\mu_n^k(x)}=
		\intRd{\abs{x+v[\mu_n^k](x)\Dt_k}^2}{\mu_n^k(x)}
\end{equation*}
and that
\begin{align*}
	\abs{x+v[\mu_n^k](x)\Dt_k}^2 &= \abs{x}^2+2\Dt_k x\cdot v[\mu_n^k](x)+\Dt_k^2\abs{v[\mu_n^k](x)}^2 \\
	&\leq \abs{x}^2+2V\Dt_k\abs{x}+V^2\Dt_k^2,
\end{align*}
therefore
\begin{equation*}
	\intRd{\abs{x}^2}{\mu_{n+1}^k(x)}\leq \intRd{\abs{x}^2}{\mu_n^k(x)}+2V\Dt_k\intRd{\abs{x}}{\mu_n^k(x)}
		+V^2\Dt_k^2.
\end{equation*}

Moreover, using the bound \eqref{bound_firstmom} we deduce
\begin{equation*}
	\intRd{\abs{x}^2}{\mu_{n+1}^k(x)}-\intRd{\abs{x}^2}{\mu_n^k(x)}\leq 2V\Dt_k\intRd{\abs{x}}{\incond{\mu}(x)}
		+2V^2T\Dt_k+V^2\Dt_k^2
\end{equation*}
and, summing telescopically over $n$,
\begin{equation*}
	\intRd{\abs{x}^2}{\mu_n^k(x)}\leq\intRd{\abs{x}^2}{\incond{\mu}(x)}+2VT\intRd{\abs{x}}{\incond{\mu}(x)}
		+3V^2T^2
\end{equation*}
whence the claims of the lemma follow also for $p=2$. \qedhere
\end{enumerate}
\end{proof}

\begin{lemma} \label{lemma:Mtk.in.CTPone}
For all $k\geq 0$ we have 
\begin{equation}
	\wass{M_s^k}{M_t^k}\leq V\abs{t-s}, \quad \forall\,s,\,t\in[0,\,T],
	\label{eq:unif.lip}
\end{equation}
\ie, the curves $M_\ccdot^k\in\CTPone$ are Lipschitz continuous uniformly in $k$.
Moreover,
\begin{equation*}
	\sup_{k\geq 0}\sup_{t\in[0,\,T]}\intRd{\abs{x}^p}{M_t^k(x)}<+\infty, \qquad p=1,\,2.
\end{equation*}
\end{lemma}
\begin{proof}
\begin{enumerate}
\item \label{MtPone} We claim $M_t^k\in\Pone$ for all $t,\,k$. To show this, we fix $k\geq 0$ and $t\in[0,\,T]$ and observe that there exists $0\leq n\leq N_k$ such that $t\in[t_n^k,\,t_{n+1}^k]$. Hence, using \eq\eqref{bound_firstmom}, we obtain
\begin{equation*}
	\intRd{\abs{x}}{M_t^k(x)}=\left(1-\frac{t-t_n^k}{\Dt_k}\right)\intRd{\abs{x}}{\mu_n^k(x)}+
		\frac{t-t_n^k}{\Dt_k}\intRd{\abs{x}}{\mu_{n+1}^k(x)}\leq
			\intRd{\abs{x}}{\incond{\mu}(x)}+VT.
\end{equation*}
From the arbitrariness of $t\in[0,\,T]$, $k\geq 0$ our claim follows, along with the uniform boundedness of the first moment of $M_t^k$ in both $t$ and $k$.

\item We prove now the estimate \eqref{eq:unif.lip}. Let $s,\,t\in[0,\,T]$ and assume, without loss of generality, that $s\leq t$. Then there exist two integers $m$ and $n$, such that $0\leq m\leq n\leq N_k$, with the property that $t_m^k\leq s\leq t_{m+1}^k$ and $t_n^k\leq t\leq t_{n+1}^k$. Therefore we can write
\begin{equation*}
	M_s^k=\left(1-\frac{s-t_m^k}{\Dt_k}\right)\mu_m^k+\frac{s-t_m^k}{\Dt_k}\mu_{m+1}^k, \quad
	M_t^k=\left(1-\frac{t-t_n^k}{\Dt_k}\right)\mu_n^k+\frac{t-t_n^k}{\Dt_k}\mu_{n+1}^k
\end{equation*}
and further, owing to the triangle inequality,
\begin{equation}
	\wass{M_s^k}{M_t^k}\leq\wass{M_s^k}{\mu_{m+1}^k}+\sum_{j=m+1}^{n-1}\wass{\mu_j^k}{\mu_{j+1}^k}+
		\wass{\mu_n^k}{M_t^k}.
	\label{distMsMt}
\end{equation}
Notice that
\begin{align*}
	\wass{M_s^k}{\mu_{m+1}^k} &= \sup_{\varphi\in\Lip}\intRd{\varphi}{(\mu_{m+1}^k-M_s^k)} \\
	&=\left(1-\frac{s-t_m^k}{\Dt_k}\right)\sup_{\varphi\in\Lip}\intRd{\varphi}{(\mu_{m+1}^k-\mu_m^k)} \\
	&=\frac{t_{m+1}^k-s}{\Dt_k}\wass{\mu_m^k}{\mu_{m+1}^k}
\end{align*}
and analogously $\wass{\mu_n^k}{M_t^k}=(t-t_n^k)\wass{\mu_n^k}{\mu_{n+1}^k}/\Dt_k$, thus, according to \eq\eqref{distMsMt}, estimating $\wass{M_s^k}{M_t^k}$ amounts to estimating $\wass{\mu_i^k}{\mu_{i+1}^k}$ for arbitrary $0\leq i\leq N_k-1$.

For any $\varphi\in\Lip$, using $\mu_{i+1}^k=\gamma_i^k\#\mu_i^k$ yields
\begin{align*}
	\intRd{\varphi(x)}{(\mu_{i+1}^k-\mu_i^k)(x)} &= \intRd{(\varphi(\gamma_i^k(x))-\varphi(x))}{\mu_i^k(x)} \\
	&\leq\intRd{\abs{\gamma_i^k(x)-x}}{\mu_i^k(x)}=\Dt_k\intRd{\abs{v[\mu_i^k](x)}}{\mu_i(x)}\leq V\Dt_k,
\end{align*}
consequently $\wass{\mu_i^k}{\mu_{i+1}^k}\leq V\Dt_k$ all $i$. From \eq\eqref{distMsMt} we deduce
\begin{align*}
	\wass{M_s^k}{M_t^k} &\leq V[(t_{m+1}^k-s)+(n-m-1)\Dt_k+(t-t_n^k)] \\
	&=V(t-s+t_{m+1}^k-t_m^k-\Dt_k)=V(t-s),
\end{align*}
which proves our claim.

\item Finally, we claim $M_t^k\in\Ptwo$ for all $t,\,k$. Arguing like in \eqref{MtPone} we have
\begin{align*}
	\intRd{\abs{x}^2}{M_t^k(x)} &= \left(1-\frac{t-t_n^k}{\Dt_k}\right)\intRd{\abs{x}^2}{\mu_n^k(x)}+
		\frac{t-t_n^k}{\Dt_k}\intRd{\abs{x}^2}{\mu_{n+1}^k(x)} \\
	&\leq \intRd{\abs{x}^2}{\incond{\mu}(x)}+2VT\intRd{\abs{x}}{\incond{\mu}(x)}+3V^2T^2,
\end{align*}
where Lemma~\ref{lemma:mu_n_k} has been used. This proves the claim and also the uniform boundedness of the second moment of $M_t^k$ in both $t$ and $k$. \qedhere
\end{enumerate}
\end{proof}

We are now ready to prove the main convergence result of this section.

\begin{proposition} \label{prop:convergence}
There exists $\mu_\ccdot\in\CTPone$ and there exists a subsequence $(M_\ccdot^{k_j})_{j\geq 0}$ such that
\begin{equation*}
	\lim_{j\to\infty}\sup_{t\in[0,\,T]}\wass{M_t^{k_j}}{\mu_t}=0.
\end{equation*}
\end{proposition}
\begin{proof}
It suffices to prove that $\{M_\ccdot^k\}_{k\geq 0}$ is a relatively compact subset of $\CTPone$. Owing to Ascoli-Arzel\`a's Theorem, this happens if $\{M_\ccdot^k\}_{k\geq 0}$ is equicontinuous and $\{M_t^k\}_{k\geq 0}$ is relatively compact in $\Pone$ for all $t\in[0,\,T]$.

\begin{enumerate}
\item Equicontinuity follows from the estimate \eqref{eq:unif.lip}. Indeed, let $\eps>0$ then for $\delta=\eps/(2V)$, which does not depend on $k$, we have $\abs{t-s}<\delta\,\Rightarrow\,\wass{M_s^k}{M_t^k}\leq\eps/2<\eps$.

\item According to \cite[Proposition 7.1.5]{ambrosio2008gfm}, the relative compactness of $\{M_t^k\}_{k\geq 0}$ in $\Pone$ is equivalent to the fact that $\{M_t^k\}_{k\geq 0}$ be tight and have uniformly integrable first moments.
\begin{enumerate}
\item Using \cite[Remark 5.1.5]{ambrosio2008gfm}, a sufficient condition for tightness is that there exists a function $\funct{\varphi}{\Rd}{[0,\,+\infty]}$, whose sublevel sets $\{x\in\Rd\,:\,\varphi(x)\leq c\}$ are compact in $\Rd$, such that
\begin{equation*}
	\sup_{k\geq 0}\intRd{\varphi}{M_t^k}<+\infty.
\end{equation*}
Taking $\varphi(x)=\abs{x}$ and invoking Lemma~\ref{lemma:Mtk.in.CTPone} we see that this condition is fulfilled, hence $\{M_t^k\}_{k\geq 0}$ is tight.

\item Using \cite[\eq (5.1.20)]{ambrosio2008gfm}, a sufficient condition for the uniform integrability of the first moments of $\{M_t^k\}_{k\geq 0}$ is that there exists $p>1$ such that
\begin{equation*}
	\sup_{k\geq 0}\intRd{\abs{x}^p}{M_t^k(x)}<+\infty.
\end{equation*}
From Lemma~\ref{lemma:Mtk.in.CTPone} we know that this actually holds for $p=2$.
\end{enumerate}
\end{enumerate}
Since $\{M_\ccdot^k\}_{k\geq 0}$ is relatively compact in $\CTPone$, up to subsequences we obtain that the sequence $(M_\ccdot^k)_{k\geq 0}$ converges in $\CTPone$ and we are done.
\end{proof}

\subsection{The limit $\bs{\mu_\ccdot}$ solves problem \eqref{eq:cauchy}} \label{sec:limit.mu}

In the previous section we have constructed a map $\mu_\ccdot$ as the limit of the sequence $(M_\ccdot^k)_{k\geq 0}$. In this section we prove that such a $\mu_\ccdot$ is a weak solution to problem \eqref{eq:cauchy}. To this end, we first derive an equation solved by the $M_t^k$'s, then we pass to the limit $k\to\infty$ thanks to Proposition~\ref{prop:convergence}.

In the following, $\phi\in\Cc$ is a fixed test function. Using the definition of $M^k_t$, let us compute
\begin{align}
	\nonumber\intRd{\phi}{M_t^k} &= \sum_{n=0}^{N_k-1}\left\{\left(1-\frac{t-t_n^k}{\Dt_k}\right)
		\intRd{\phi}{\mu_n^k}+\frac{t-t_n^k}{\Dt_k}\intRd{\phi}{\mu_{n+1}^k}\right\}\ind{[t_n^k,\,t_{n+1}^k]}(t) \\
	&=\sum_{n=0}^{N_k-1}\left\{\intRd{\phi}{\mu_n^k}+\frac{t-t_n^k}{\Dt_k}
		\intRd{(\phi\circ\gamma_n^k-\phi)}{\mu_n^k}\right\}\ind{[t_n^k,\,t_{n+1}^k]}(t).
		\label{eq:phi-against}
\end{align}
A Taylor expansion of $\phi\circ\gamma_n^k$ with Lagrange's reminder gives
\begin{align*}
	\phi(\gamma_n^k(x)) &= \phi(x+v[\mu_n^k](x)\Dt_k) \\
	&=\phi(x)+\nabla\phi(x)\cdot v[\mu_n^k](x)\Dt_k+
		\frac{1}{2}(D^2\phi(\xlag)v[\mu_n^k](x))\cdot v[\mu_n^k](x)\Dt_k^2,
\end{align*}
where $D^2\phi$ is the Hessian of $\phi$ and $\xlag$ is a point of the segment connecting $x$ and $x+v[\mu_n^k](x)\Dt_k$. Hence the previous computation specializes as
\begin{align*}
	\intRd{\phi}{M_t^k} &= \sum_{n=0}^{N_k-1}\left\{\intRd{\phi}{\mu_n^k}+(t-t_n^k)
		\intRd{v[\mu_n^k]\cdot\nabla\phi}{\mu_n^k}\right. \notag \\
	&\phantom{=} \left.+\frac{1}{2}\Dt_k(t-t_n^k)
			\intRd{(D^2\phi(\xlag)v[\mu_n^k])\cdot v[\mu_n^k]}{\mu_n^k}\right\}\ind{[t_n^k,\,t_{n+1}^k]}(t).
\end{align*}

We claim now that the mapping $t\mapsto \int_{\Rd}\phi\,dM_t^k$ is Lipschitz continuous, hence a.e. differentiable by Rademacher's Theorem. To see this, observe that $x\mapsto\phi(x)/\lip{\phi}$ is Lipschitz continuous with Lipschitz constant at most $1$, so that
\begin{align*}
	\abs{\intRd{\phi}{M_t^k}-\intRd{\phi}{M_s^k}} &=
		\lip{\phi}\abs{\intRd{\frac{\phi}{\lip{\phi}}}{(M_t^k-M_s^k)}} \\
	&\leq\lip{\phi}\wass{M_s^k}{M_t^k}\leq\lip{\phi}V\abs{t-s}.
\end{align*}
Thus, using \eq\eqref{eq:phi-against}, we compute the derivative
\begin{equation}
	\frac{d}{dt}\intRd{\phi}{M_t^k}=\sum_{n=0}^{N_k-1}\left\{\intRd{v[\mu_n^k]\cdot\nabla\phi}{\mu_n^k}
		+\frac{1}{2}\Dt_k\intRd{(D^2\phi(\xlag)v[\mu_n^k])\cdot v[\mu_n^k]}{\mu_n^k}\right\}
			\ind{[t_n^k,\,t_{n+1}^k]}(t).
	\label{eq:deriv.Mtk}
\end{equation}

Let us consider now:
\begin{align*}
	\intRd{v[M_t^k]\cdot\nabla\phi}{M_t^k} &= \sum_{n=0}^{N_k-1}\left\{\left(1-\frac{t-t_n^k}{\Dt_k}\right)
		\intRd{v[M_t^k]\cdot\nabla\phi}{\mu_n^k}\right. \\
	&\phantom{=}+\left.\frac{t-t_n^k}{\Dt_k}\intRd{v[M_t^k]\cdot\nabla\phi}{\mu_{n+1}^k}
		\right\}\ind{[t_n^k,\,t_{n+1}^k]}(t)
\intertext{and invoke Assumption~\ref{ass:prop-v}-\eqref{item:ass-lin} to get}
	&=\sum_{n=0}^{N_k-1}\intRd{v[\mu_n^k]\cdot\nabla\phi}{\mu_n^k}\ind{[t_n^k,\,t_{n+1}^k]}(t) \\
	&\phantom{=}+\sum_{n=0}^{N_k-1}\left\{\frac{t-t_n^k}{\Dt_k}
		\intRd{v[\mu_n^k]\cdot\nabla\phi}{(\mu_{n+1}^k-\mu_n^k)}\right. \\
	&\phantom{=}-\left(\frac{t-t_n^k}{\Dt_k}\right)^2
		\intRd{(v[\mu_{n+1}^k]-v[\mu_n^k])\cdot\nabla\phi}{\mu_n^k} \\
	&\phantom{=}+\left.\left(\frac{t-t_n^k}{\Dt_k}\right)^2
		\intRd{(v[\mu_{n+1}^k]-v[\mu_n^k])\cdot\nabla\phi}{\mu_{n+1}^k}\right\}\ind{[t_n^k,\,t_{n+1}^k]}(t).
\end{align*}
Using this in \eqref{eq:deriv.Mtk} gives
\begin{align}
	\frac{d}{dt}\intRd{\phi}{M_t^k}-\intRd{v[M_t^k]\cdot\nabla\phi}{M_t^k} &=
		\sum_{n=0}^{N_k-1}\left\{\frac{1}{2}\Dt_k
			\intRd{(D^2\phi(\xlag)v[\mu_n^k])\cdot v[\mu_n^k]}{\mu_n^k}\right. \notag \\
	&\phantom{=}-\frac{t-t_n^k}{\Dt_k}\intRd{v[\mu_n^k]\cdot\nabla\phi}{(\mu_{n+1}^k-\mu_n^k)} \notag \\
	&\phantom{=}+\left(\frac{t-t_n^k}{\Dt_k}\right)^2\intRd{(v[\mu_{n+1}^k]-v[\mu_n^k])\cdot\nabla\phi}{\mu_n^k}
		\notag \\
	&\phantom{=}-\left.\left(\frac{t-t_n^k}{\Dt_k}\right)^2
		\intRd{(v[\mu_{n+1}^k]-v[\mu_n^k])\cdot\nabla\phi}{\mu_{n+1}^k}\right\}\ind{[t_n^k,\,t_{n+1}^k]}(t).
	\label{eq:eq-for-Mtk}		
\end{align}
Formally we can regard this expression as an equation satisfied by $M_t^k$. The remaining of this section is devoted to relate this equation to \eq\eqref{eq:weak.sol}.

\smallskip

Set $H:=\infnorm{D^2\phi}=\lip{\nabla\phi}$ and notice that, owing to Assumption~\ref{ass:prop-v}-\eqref{item:ass-bound}, \eqref{item:ass-lip}, the function $x\mapsto v[\mu_n^k](x)\cdot\nabla\phi(x)$ is Lipschitz continuous with
\begin{equation}
	\lip{v\cdot\nabla\phi}\leq HV+\lip{v}\infnorm{\nabla\phi}=:L.
	\label{eq:lip.v.nablaphi}
\end{equation}
Consequently:
\begin{align*}
	\abs{\frac{d}{dt}\intRd{\phi}{M_t^k}-\intRd{v[M_t^k]\cdot\nabla\phi}{M_t^k}} &\leq
		\sum_{n=0}^{N_k-1}\Biggl\{\frac{1}{2}HV^2\Dt_k
			+\frac{t-t_n^k}{\Dt_k}L\abs{\intRd{\frac{v[\mu_n^k]\cdot\nabla\phi}{L}}{(\mu_{n+1}^k-\mu_n^k)}} \\
	&\phantom{=}+2\left(\frac{t-t_n^k}{\Dt_k}\right)^2\infnorm{\nabla\phi}\wass{\mu_n^k}{\mu_{n+1}^k}
		\Biggr\}\ind{[t_n^k,\,t_{n+1}^k]}(t) \\
	&\leq C\sum_{n=0}^{N_k-1}\Biggl\{\Dt_k+\frac{t-t_n^k}{\Dt_k}\wass{\mu_n^k}{\mu_{n+1}^k} \\
	&\phantom{=}+\left(\frac{t-t_n^k}{\Dt_k}\right)^2\wass{\mu_n^k}{\mu_{n+1}^k}\Biggr\}\ind{[t_n^k,\,t_{n+1}^k]}(t)
\intertext{where $C:=\max\{\frac{1}{2}HV^2,\,L,\,2\infnorm{\nabla\phi}\}$. But $\wass{\mu_n^k}{\mu_{n+1}^k}\leq V\Dt_k$, therefore}
	&\leq C\sum_{n=0}^{N_k-1}\left\{\Dt_k+V(t-t_n^k)+V\frac{(t-t_n^k)^2}{\Dt_k}\right\}\ind{[t_n^k,\,t_{n+1}^k]}(t).
\intertext{Now, for $t\in[t_n^k,\,t_{n+1}^k]$ it results $t-t_n^k\leq\Dt_k$. In addition, $\sum_{n=0}^{N_k-1}\ind{[t_n^k,\,t_{n+1}^k]}(t)=\ind{[0,\,T]}(t)$, hence finally}
	&\leq C(1+2V)\Dt_k\ind{[0,\,T]}(t).
\end{align*}

Integrating \eq\eqref{eq:eq-for-Mtk} between $0$ and $t\leq T$ and using the last inequality yields
\begin{align*}
	\abs{\lint_0^t\left(\frac{d}{d\tau}
		\intRd{\phi}{M_\tau^k}-\intRd{v[M_\tau^k]\cdot\nabla\phi}{M_\tau^k}\right)\,d\tau} &\leq	
	\lint_0^t\abs{\frac{d}{d\tau}\intRd{\phi}{M_\tau^k}-\intRd{v[M_\tau^k]\cdot\nabla\phi}{M_\tau^k}}\,d\tau \\
	&\leq C(1+2V)T\Dt_k
\end{align*}
so that, by further manipulating the left-hand side and taking the limit for $k\to\infty$, we obtain
\begin{equation}
	\lim_{k\to\infty}\abs{\intRd{\phi}{M_t^k}-\intRd{\phi}{\incond{\mu}}-
		\lint_0^t\intRd{v[M_\tau^k]\cdot\nabla\phi}{M_\tau^k}\,d\tau}=0.
	\label{eq:Mtk.limit}
\end{equation}

To infer from \eqref{eq:Mtk.limit} that $\mu_t$ solves \eqref{eq:weak.sol}, we need the following convergence result.
\begin{lemma} \label{lemma:conv.int}
When $k\to\infty$ we have, up to subsequences,
\begin{equation*}
	\intRd{\phi}{M_t^k}\to\intRd{\phi}{\mu_t} \qquad \text{and} \qquad
	\lint_0^t\intRd{v[M_\tau^k]\cdot\nabla\phi}{M_\tau^k}\,d\tau\to
		\lint_0^t\intRd{v[\mu_\tau]\cdot\nabla\phi}{\mu_\tau}\,d\tau
\end{equation*}
for all $t\in[0,\,T]$.
\end{lemma}
\begin{proof}
\begin{enumerate}
\item For the first limit we write
\begin{align*}
	\abs{\intRd{\phi}{\mu_t}-\intRd{\phi}{M_t^k}} &= \abs{\intRd{\phi}{(\mu_t-M_t^k)}} \\
	&=\lip{\phi}\abs{\intRd{\frac{\phi}{\lip{\phi}}}{(\mu_t-M_t^k)}}\\
	&\leq\lip{\phi}\wass{M_t^k}{\mu_t},
\end{align*}
and, up to passing to a suitable subsequence of $(M_\ccdot^k)_{k\geq 0}$, we conclude by applying Proposition~\ref{prop:convergence}.

\item For the second limit we preliminarily observe that $\vert\int_{\Rd}v[M_\tau^k]\cdot\nabla\phi\,dM_\tau^k\vert\leq V\infnorm{\nabla\phi}$, thus, by dominated convergence,
\begin{equation}
	\lim_{k\to\infty}\lint_0^t\intRd{v[M_\tau^k]\cdot\nabla\phi}{M_\tau^k}\,d\tau=
		\lint_0^t\left(\lim_{k\to\infty}\intRd{v[M_\tau^k]\cdot\nabla\phi}{M_\tau^k}\right)\,d\tau.
	\label{eq:dominated}
\end{equation}
Now
\begin{align*}
	\abs{\intRd{v[\mu_\tau]\cdot\nabla\phi}{\mu_\tau}-\intRd{v[M_\tau^k]\cdot\nabla\phi}{M_\tau^k}} &\leq
		\intRd{\abs{v[\mu_\tau]-v[M_\tau^k]}\cdot\abs{\nabla\phi}}{\mu_\tau} \\
	&\phantom{\leq}+\abs{\intRd{v[M_\tau^k]\cdot\nabla\phi}{(\mu_\tau-M_\tau^k)}} \\
	&\leq(\infnorm{\nabla\phi}\lip{v}+L)\wass{M_\tau^k}{\mu_\tau},
\end{align*}
$L$ being the constant defined in \eq\eqref{eq:lip.v.nablaphi}. Proposition~\ref{prop:convergence} implies then $\int_{\Rd}v[M_\tau^k]\cdot\nabla\phi\,dM_\tau^k\to\int_{\Rd}v[\mu_\tau]\cdot\nabla\phi\,d\mu_\tau$ for all $\tau\in[0,\,T]$, and the thesis follows from \eq\eqref{eq:dominated}. \qedhere
\end{enumerate}
\end{proof}

Combining \eq\eqref{eq:Mtk.limit} and Lemma~\ref{lemma:conv.int}, and thanks to the arbitrariness of $\phi$, we obtain that $\mu_\ccdot$ solves problem \eqref{eq:cauchy} in the sense of Definition~\ref{def:weak.sol}. In conclusion, we have proved:
\begin{theorem}[Existence] \label{theo:existence}
Let Assumptions~\ref{ass:prop-v},~\ref{ass:in.cond} hold. Then there exists a weak solution $\mu_\ccdot\in\CTPone$ to the Cauchy problem \eqref{eq:cauchy}.
\end{theorem}

\section{Approximation of the solutions} \label{sec:approx}
This section is devoted to a convergence analysis of the numerical scheme proposed in \cite{piccoli2010tem} for the approximation of the solutions to problem \eqref{eq:cauchy}. We begin by sketching the main ideas which underlie the construction of the scheme, referring the reader to the above-cited paper for a detailed derivation. The scheme is obtained from a twofold approximation of problem \eqref{eq:cauchy}, in time and in space.
To this goal, we introduce a discretization of the time interval $[0,\,T]$ by means of discrete instants $t_n^k=n\Dt_k$, where the index $n$ ranges from $0$ to a value $N_k$ such that $N_k\Dt_k=T$ and the time step $\Dt_k>0$ tends to $0$ when $k\to\infty$. By this discretization, we obtain the discrete-time dynamical system \eqref{eq:pushfwd}. Then, we introduce a space discretization in the following way. We define a pairwise disjoint partition of $\Rd$ made of measurable elements $E_i^k\in\B$, $i=(i_1,\,\dots,\,i_d)\in\Zd$, $k\geq 0$, such that
\begin{equation*}
	\bigcup_{i\in\Zd}E_i^k=\Rd, \qquad	E_i^k\cap E_j^k=\emptyset \quad\forall\,i\ne j
\end{equation*}
for all $k\geq 0$.
For the sake of simplicity, we assume that the $E_i^k$'s are hypercubes of edge length $h_k>0$, with $h_k\to 0$ when $k\to\infty$. Specifically,
\begin{equation*}
	E_i^k=\bigcart_{l=1}^{d}\left[i_l-\frac{1}{2} ,\,i_l+\frac{1}{2}\right)h_k.
\end{equation*}
Note that the index $k$ identifies the level of refinement of the numerical grid. 
Using this space discretization, we approximate both the initial condition $\incond{\mu}$ and the measures $\mu_n^k$ by means of piecewise constant measures $\lambda_n^k\ll\leb^d$, $n=0,\,\dots,\,N_k$, $k\geq 0$. More precisely, $d\lambda_n^k=\rho_n^k\,dx$ with
\begin{equation*}
	\rho_n^k(x)=\sum_{i\in\Zd}\rho_i^n\ind{E_i^k}(x).
\end{equation*}
The space discretization makes it necessary to approximate also the flow map $\gamma_n^k$ defined in \eq\eqref{eq:flowmap}. This is accomplished by the mapping
\begin{equation*}
	\appr{\gamma}_n^k(x)=x+\appr{v}_n^k(x)\Dt_k, \qquad
		\appr{v}_n^k(x)=\sum_{i\in\Zd}v[\lambda_n^k](x_i^k)\ind{E_i^k}(x),
\end{equation*}
$x_i^k$ being a point of the grid cell $E_i^k$ (\eg, its center $x_i^k=ih_k$). In practice, the velocity $v[\mu_n^k](x)$ is approximated by the piecewise constant field $\appr{v}_n^k(x)$ taking in $E_i^k$ the value that $v$, computed \wrt the measure $\lambda_n^k$, takes in $x_i^k$. Notice that $\abs{\appr{v}_n^k(x)}\leq V$ for all $x\in\Rd$, all $n=0,\,\dots,\,N_k$, and all $k\geq 0$ because of Assumption~\ref{ass:prop-v}-\eqref{item:ass-bound}.

By imposing $\lambda_{n+1}^k(E_i^k)=(\appr{\gamma}_n^k\#\lambda_n^k)(E_i^k)$ for each $i\in\Zd$, we deduce the following explicit-in-time scheme relating recursively the coefficients $\{\rho_i^n\}$ at two successive time steps:
\begin{equation}
	\rho_i^{n+1}=\frac{1}{h_k^d}\sum_{j\in\Zd}\rho_j^n\leb^d(E_j^k\cap(\appr{\gamma}_n^k)^{-1}(E_i^k)),
	\label{eq:numscheme}
\end{equation}
where $h_k^d$ is $\leb^d(E_i^k)$. To start up the scheme one has to provide the coefficients $\rho_i^0$, that we obtain as
\begin{equation}
	\rho_i^0=\frac{1}{h_k^d}\incond{\mu}(E_i^k), \quad i\in\Zd.
	\label{eq:numscheme.incond}
\end{equation}

\bigskip

We are now ready to start the analysis of the proposed scheme. To this end, we begin with a simple property of the approximation measures.
\begin{lemma} \label{lemma:lambdank-proba}
For every $n=0,\,\dots,\,N_k$, $k\geq 0$, the $\lambda_n^k$'s are probability measures. 
\end{lemma}
\begin{proof}
\begin{enumerate}
\item The claim is certainly true for $\lambda_0^k$, since \eq\eqref{eq:numscheme.incond} shows that $\rho_i^0\geq 0$ all $i$ and furthermore
\begin{equation*}
	\lambda_0^k(\Rd)=\intRd{\rho_0^k(x)}{x}=h_k^d\sum_{i\in\Zd}\rho_i^0=\sum_{i\in\Zd}\incond{\mu}(E_i^k)
		=\incond{\mu}(\Rd)=1,
\end{equation*}
where we have used the $\sigma$-additivity of $\incond{\mu}$.
\item If we assume now that $\lambda_n^k$ is a probability measure for a certain $n$, using \eq\eqref{eq:numscheme} we get $\rho_i^{n+1}\geq 0$ all $i$ and moreover
\begin{align*}
	\lambda_{n+1}^k(\Rd) &= h_k^d\sum_{i\in\Zd}\rho_i^{n+1}=
		\sum_{i\in\Zd}\sum_{j\in\Zd}\rho_j^n\leb^d(E_j^k\cap(\appr{\gamma}_n^k)^{-1}(E_i^k)) \\
	&=\sum_{j\in\Zd}\rho_j^n\sum_{i\in\Zd}\leb^d(E_j^k\cap(\appr{\gamma}_n^k)^{-1}(E_i^k))
		=\sum_{j\in\Zd}\rho_j^n\leb^d(E_j^k\cap(\appr{\gamma}_n^k)^{-1}(\Rd)) \\
	&=h_k^d\sum_{j\in\Zd}\rho_j^n=\lambda_n^k(\Rd)=1,
\end{align*}
where we have used the fact that $(\appr{\gamma}_n^k)^{-1}(E_i^k)$ are pairwise disjoint and that set operations, \eg, union, commute with the inverse image of a function. By induction on $n$, and by the arbitrariness of $k\geq 0$, the claim follows. \qedhere
\end{enumerate}
\end{proof}

The next result shows that the link between two successive measures $\lambda_n^k$, $\lambda_{n+1}^k$ is a push forward as defined by \eq\eqref{eq:def.pushfwd}, provided we restrict test functions to \emph{simple functions} adapted to the spatial grid $\{E_i^k\}_{i\in\Zd}$, \ie, piecewise constant functions $\funct{s}{\Rd}{\R}$ of the form
\begin{equation*}
	s(x)=\sum_{i\in\Zd}\alpha_i\ind{E_i^k}(x) \qquad (\alpha_i\in\R).
\end{equation*}

\begin{lemma} \label{lemma:s}
Let $\funct{s}{\Rd}{\R}$ be simple over the grid $\{E_i^k\}_{i\in\Zd}$. Then
\begin{equation*}
	\intRd{s}{\lambda_{n+1}^k}=\intRd{s\circ\appr{\gamma}_n^k}{\lambda_n^k}.
\end{equation*}
\begin{proof}
We have
\begin{align*}
	\intRd{s}{\lambda_{n+1}^k}=h_k^d\sum_{i\in\Zd}\alpha_i\rho_i^{n+1} &=
			\sum_{i\in\Zd}\alpha_i\sum_{j\in\Zd}\rho_j^n\leb^d(E_j^k\cap(\appr{\gamma}_n^k)^{-1}(E_i^k)) \\
	&= \sum_{j\in\Zd}\rho_j^n\sum_{i\in\Zd}\alpha_i\leb^d(E_j^k\cap(\appr{\gamma}_n^k)^{-1}(E_i^k)).
\end{align*}
Notice that $s(\appr{\gamma}_n^k(x))=\sum_{i\in\Zd}\alpha_i\ind{(\appr{\gamma}_n^k)^{-1}(E_i^k)}(x)$, \ie, $s\circ\appr{\gamma}_n^k$ is piecewise constant over the sets $(\appr{\gamma}_n^k)^{-1}(E_i^k)$, which form a pairwise disjoint partition of $\Rd$. Therefore, for any given $A\in\B$,
\begin{equation*}
	\intg{A}{s(\appr{\gamma}_n^k(x))}{x}=\sum_{i\in\Zd}
		\intg{A\cap(\appr{\gamma}_n^k)^{-1}(E_i^k)}{s(\appr{\gamma}_n^k(x))}{x}=
			\sum_{i\in\Zd}\alpha_i\leb^d(A\cap(\appr{\gamma}_n^k)^{-1}(E_i^k)).
\end{equation*}
For $A=E_j^k$ this enables us to continue the previous computation as
\begin{equation*}
	\intRd{s}{\lambda_{n+1}^k}=\sum_{j\in\Zd}\rho_j^n\intg{E_j^k}{s(\appr{\gamma}_n^k(x))}{x}=
		\intRd{s\circ\appr{\gamma}_n^k}{\lambda_n^k}
\end{equation*}
and to obtain the thesis.
\end{proof}
\end{lemma}

An immediate consequence of this lemma is the following result, that will be fundamental for the sequel.
\begin{lemma} \label{lemma:approx.pushfwd}
For all Lipschitz continuous $\varphi:\Rd\to\R$ we have
\begin{equation*}
	\abs{\intRd{\varphi}{\lambda_{n+1}^k}-\intRd{\varphi\circ\appr{\gamma}_n^k}{\lambda_n^k}}\leq
		2\lip{\varphi}\sqrt{d}h_k, \quad \forall\,n=0,\,\dots,\,N_k.
\end{equation*}
\end{lemma}
\begin{proof}
We consider the simple function
\begin{equation*}
	s(x)=\sum_{i\in\Zd}\varphi(x_i^k)\ind{E_i^k}(x)
\end{equation*}
and, using Lemma~\ref{lemma:s}, we compute:
\begin{align*}
	\intRd{\varphi(x)}{\lambda_{n+1}^k(x)} &= \intRd{(\varphi(x)-s(x))}{\lambda_{n+1}^k}+
		\intRd{s(x)}{\lambda_{n+1}^k(x)} \\
	&= \intRd{(\varphi(x)-s(x))}{\lambda_{n+1}^k(x)}+
		\intRd{[s(\appr{\gamma}_n^k(x))-\varphi(\appr{\gamma}_n^k(x))]}{\lambda_n^k(x)} \\
	&\phantom{=}+\intRd{\varphi(\appr{\gamma}_n^k(x))}{\lambda_n^k(x)} \\
	&= \sum_{i\in\Zd}\left\{\intg{E_i^k}{(\varphi(x)-\varphi(x_i^k))}{\lambda_{n+1}^k(x)}+
		\intg{E_i^k}{(\varphi(x_i^k)-\varphi(x))}{(\appr{\gamma}_n^k\#\lambda_n^k)(x)}\right\} \\
	&\phantom{=}+\intRd{\varphi(\appr{\gamma}_n^k(x))}{\lambda_n^k(x)}.
\end{align*}
Notice that $\appr{\gamma}_n^k\#\lambda_n^k$ is a probability measure. The Lipschitz continuity of $\varphi$ entails
\begin{align*}
	& \abs{\intRd{\varphi(x)}{\lambda_{n+1}^k(x)}-\intRd{\varphi(\appr{\gamma}_n^k(x))}{\lambda_n^k(x)}} \\
	& \ \leq\lip{\varphi}\sum_{i\in\Zd}\left\{\intg{E_i^k}{\abs{x-x_i^k}}{\lambda_{n+1}^k(x)}+
		\intg{E_i^k}{\abs{x_i^k-x}}{(\appr{\gamma}_n^k\#\lambda_n^k)(x)}\right\}
\intertext{whence, since $\abs{x-x_i^k}\leq\diam{E_i^k}=\sqrt{d}h_k$ for all $x\in E_i^k$,}
	& \ \leq\lip{\varphi}\sqrt{d}h_k\sum_{i\in\Zd}\left\{\lambda_{n+1}^k(E_i^k)+
		(\appr{\gamma}_n^k\#\lambda_n^k)(E_i^k)\right\}=2\lip{\varphi}\sqrt{d}h_k
\end{align*}
and, due to the arbitrariness of $\varphi$ and $n$, the thesis follows.
\end{proof}

To address the convergence of the measures $\lambda_n^k$, it is convenient to introduce the following linear-in-time interpolation:
\begin{equation}
	\Lambda_t^k=\sum_{n=0}^{N_k-1}\left[\left(1-\frac{t-t_n^k}{\Dt_k}\right)\lambda_n^k+
		\frac{t-t_n^k}{\Dt_k}\lambda_{n+1}^k\right]\ind{[t_n^k,\,t_{n+1}^k]}(t),
	\label{eq:Lambda.interp}
\end{equation}
with $\Lambda_t^k$ a probability measure for all $t\in[0,\,T]$ and all $k\geq 0$. This definition is reminiscent of the definition of $M_t^k$ (cf. Section \ref{sec:existence}). 

In order to prove a convergence result about $\Lambda_t^k$, from now on we make the following assumption:
\begin{assumption} \label{ass:mesh.param}
The mesh parameters $h_k,\,\Dt_k$ satisfy
\begin{equation*}
	h_k=o(\Dt_k) \quad \text{for\ } k\to\infty.
\end{equation*}
\end{assumption}
As a consequence, there exists a sequence $(\beta_k)_{k\geq 0}\subset\R$, with $\beta_k>0$ all $k$ and $\lim_{k\to\infty}\beta_k=0$, such that $h_k=\beta_k\Dt_k$. By convergence, it further results $\sup_{k\geq 0}h_k,\,\sup_{k\geq 0}\Dt_k,\,\sup_{k\geq 0}\beta_k<+\infty$.

\subsection{Regularity of $\bs{\Lambda_\ccdot^k}$}
This section is devoted to establish some regularity properties of the curves $\Lambda_\ccdot^k$.

\begin{lemma} \label{lemma:lambda_n_k}
We have $\lambda_n^k\in\Pone$ for all $n=0,\,\dots,\,N_k$ and all $k\geq 0$. In particular,
\begin{equation*}
	\sup_{k\geq 0}\sup_{0\leq n\leq N_k}\intRd{\abs{x}}{\lambda_n^k(x)}<+\infty,
\end{equation*}
\ie, first order moments of the $\lambda_n^k$'s are uniformly bounded.
\end{lemma}
\begin{proof}
\begin{enumerate}
\item \label{lambda_0_k} We begin by proving an inequality for $n=0$. Recalling \eq\eqref{eq:numscheme.incond}, we have
\begin{equation*}
	\intRd{\abs{x}}{\lambda_0^k(x)}=\sum_{i\in\Zd}\rho_i^0\intg{E_i^k}{\abs{x}}{x}=
		\sum_{i\in\Zd}\frac{1}{h_k^d}\intg{E_i^k}{\abs{x}}{x}\,\incond{\mu}(E_i^k)=
			\intRd{s(x)}{\incond{\mu}(x)},
\end{equation*}
where $s$ is the simple function $s=\sum_{i\in\Zd}\alpha_i\ind{E_i^k}$ with $\alpha_i=\frac{1} {h_k^d}\int_{E_i^k}\abs{x}\,dx$. Thus
\begin{align*}
	\abs{\intRd{\abs{x}}{\lambda_0^k(x)}-\intRd{\abs{x}}{\incond{\mu}(x)}} &\leq
		\intRd{\abs{s(x)-\abs{x}}}{\incond{\mu}(x)}\\
		=&\sum_{i\in\Zd}\intg{E_i^k}{\abs{\alpha_i-\abs{x}}}{\incond{\mu}(x)}.
\intertext{Considering that $\abs{\alpha_i-\abs{x}}=\abs{\frac{1}{h_k^d}\int_{E_i^k}(\abs{y}-\abs{x})\,dy}\leq\frac{1}{h_k^d}\int_{E_i^k}\abs{x-y}\,dy$, we further deduce}
	&\leq\sum_{i\in\Zd}\frac{1}{h_k^d}\intg{E_i^k}{\intg{E_i^k}{\abs{x-y}}{y}}{\incond{\mu}(x)}\leq
		\sqrt{d}h_k,
\end{align*}
where we have used that, in the double integral, both $x$ and $y$ are points of the same grid cell $E_i^k$, hence $\abs{x-y}\leq\diam{E_i^k}=\sqrt{d}h_k$. It follows
\begin{equation*}
	\intRd{\abs{x}}{\lambda_0^k(x)}\leq\intRd{\abs{x}}{\incond{\mu}(x)}+\sqrt{d}h_k,
\end{equation*}
whence the thesis for $n=0$.

\item We obtain the general case $0<n\leq N_k$ by induction from this and Lemma~\ref{lemma:approx.pushfwd}. Choosing $\varphi(x)=\abs{x}$ in the latter yields
\begin{align*}
	\intRd{\abs{x}}{\lambda_{n+1}^k(x)} &\leq \intRd{\abs{\appr{\gamma}_n^k(x)}}{\lambda_n^k(x)}+2\sqrt{d}h_k \\
	&\leq\intRd{\abs{x}}{\lambda_n^k(x)}+\Dt_k\intRd{\abs{\appr{v}_n^k(x)}}{\lambda_n^k(x)}+2\sqrt{d}h_k \\
	&\leq\intRd{\abs{x}}{\lambda_n^k(x)}+V\Dt_k+2\sqrt{d}h_k
\end{align*}
whence, summing telescopically and using $h_k=\beta_k\Dt_k$,
\begin{align*}
	\intRd{\abs{x}}{\lambda_n^k(x)} &\leq \intRd{\abs{x}}{\lambda_0^k(x)}+n(V\Dt_k+2\sqrt{d}h_k) \\
	&\leq\intRd{\abs{x}}{\incond{\mu}(x)}+(V+2\sqrt{d}\beta_k)t_n^k+\sqrt{d}h_k.
\end{align*}
The thesis now follows taking the supremum of both sides in $n$ and $k$ while considering that $t_n^k\leq T$. \qedhere
\end{enumerate}
\end{proof}

\begin{lemma} \label{lemma:Lambdatk.in.CTPone}
For all $k\geq 0$, 
\begin{equation}
	\wass{\Lambda_s^k}{\Lambda_t^k}\leq (V+2\sqrt{d}\bar{\beta})\abs{t-s},
	\label{eq:unif.lip.2}
\end{equation}
where $\bar{\beta}:=\sup_{k\geq 0}\beta_k$, \ie, the curves $\Lambda_\ccdot^k\in\CTPone$ are Lipschitz continuous uniformly in $k$. Moreover,
\begin{equation*}
	\sup_{k\geq 0}\sup_{t\in[0,\,T]}\intRd{\abs{x}}{\Lambda_t^k(x)}<+\infty.
\end{equation*}
\end{lemma}
\begin{proof}
\begin{enumerate}
\item Fix $k\geq 0$ and $t\in[0,\,T]$. There exists $0\leq n\leq N_k$ such that $t\in[t_n^k,\,t_{n+1}^k]$, hence
\begin{align*}
	\intRd{\abs{x}}{\Lambda_t^k(x)} &= \left(1-\frac{t-t_n^k}{\Dt_k}\right)\intRd{\abs{x}}{\lambda_n^k(x)}+
		\frac{t-t_n^k}{\Dt_k}\intRd{\abs{x}}{\lambda_{n+1}^k(x)}.
\intertext{Owing to Lemma~\ref{lemma:lambda_n_k}, we can find a uniform upper bound on the first moments of the $\lambda_n^k$'s: there exists a constant $C>0$, independent of $n$ and $k$, such that}
	&\leq \left(1-\frac{t-t_n^k}{\Dt_k}\right)C+\frac{t-t_n^k}{\Dt_k}C=C,
\end{align*}
which says that $\Lambda_t^k\in\Pone$ for all $k\geq 0$ and all $t\in[0,\,T]$ with uniformly bounded first moment.

\item We argue the Lipschitz continuity of the mapping $t\mapsto\Lambda_t^k$ as in Lemma~\ref{lemma:Mtk.in.CTPone}. In particular, for $s,\,t\in[0,\,T]$ with $s\leq t$ we obtain
\begin{equation}
	\wass{\Lambda_s^k}{\Lambda_t^k}\leq\frac{t_{m+1}^k-s}{\Dt_k}\wass{\lambda_m^k}{\lambda_{m+1}^k}+
		\sum_{j=m+1}^{n-1}\wass{\lambda_j^k}{\lambda_{j+1}^k}+
			\frac{t-t_n^k}{\Dt_k}\wass{\lambda_n^k}{\lambda_{n+1}^k},
	\label{distLambdasLambdat}
\end{equation}
where $0\leq m\leq n\leq N_k$ are such that $s\in[t_m^k,\,t_{m+1}^k]$ and $t\in[t_n^k,\,t_{n+1}^k]$. In order to estimate $\wass{\lambda_i^k}{\lambda_{i+1}^k}$ for a generic $0\leq i\leq N_k$ we fix $\varphi\in\Lip$ and use Lemma~\ref{lemma:approx.pushfwd}:
\begin{align*}
	\intRd{\varphi}{(\lambda_{i+1}^k-\lambda_i^k)} &\leq
		\intRd{[\varphi(\appr{\gamma}_n^k(x))-\varphi(x)]}{\lambda_n^k(x)}+2\sqrt{d}h_k \\
	&\leq \intRd{\abs{\appr{\gamma}_n^k(x)-x}}{\lambda_n^k(x)}+2\sqrt{d}h_k
		=\Dt_k\intRd{\abs{\appr{v}_n^k(x)}}{\lambda_n^k(x)}+2\sqrt{d}h_k \\
	&\leq V\Dt_k+2\sqrt{d}h_k=(V+2\sqrt{d}\beta_k)\Dt_k.
\end{align*}
Consequently $\wass{\lambda_i^k}{\lambda_{i+1}^k}\leq C\Dt_k$ all $i$, where $C:=V+2\sqrt{d}\sup_{k\geq 0}\beta_k$. Computing as in Lemma~\ref{lemma:Mtk.in.CTPone}, in view of \eq\eqref{distLambdasLambdat} this yields $\wass{\Lambda_s^k}{\Lambda_t^k}\leq C(t-s)$ and we have the thesis. \qedhere
\end{enumerate}
\end{proof}

\subsection{Limit equation for $\bs{\Lambda_\ccdot^k}$}
The next step is to find, similarly to what we did in Section~\ref{sec:limit.mu}, an equation satisfied by $\Lambda_\ccdot^k$ in which to pass to the limit $k\to\infty$. Fix a test function $\phi\in\Cc$ and notice that the mapping $t\mapsto\int_{\Rd}\phi\,d\Lambda_t^k$ is Lipschitz continuous because so is the mapping $t\mapsto\Lambda_t^k$ in view of Lemma~\ref{lemma:Lambdatk.in.CTPone}. Then, owing to Rademacher's Theorem, it is a.e. differentiable. Using expression \eqref{eq:Lambda.interp}, we find that its derivative is
\begin{equation}
	\frac{d}{dt}\intRd{\phi}{\Lambda_t^k}=\frac{1}{\Dt_k}\sum_{n=0}^{N_k-1}
		\intRd{\phi}{(\lambda_{n+1}^k-\lambda_n^k)}\ind{[t_n^k,\,t_{n+1}^k]}(t).
	\label{eq:deriv.Lambda}
\end{equation}

Let us introduce now the function $g_n^k:\R^d\to\Rd$,
\begin{equation*}
	g_n^k(x)=x+v[\lambda_n^k](x)\Dt_k,
\end{equation*}
\ie, the flow map $\gamma_n^k$ (cf. \eq\eqref{eq:flowmap}) computed \wrt the measure $\lambda_n^k$ instead of $\mu_n^k$. Then
\begin{align*}
	\frac{d}{dt}\intRd{\phi}{\Lambda_t^k} &= \frac{1}{\Dt_k}\sum_{n=0}^{N_k-1}
		\left\{\intRd{\phi}{\lambda_{n+1}^k}-\intRd{\phi\circ g_n^k}{\lambda_n^k}\right\}
			\ind{[t_n^k,\,t_{n+1}^k]}(t) \\
	&\phantom{=} +\frac{1}{\Dt_k}\sum_{n=0}^{N_k-1}\intRd{(\phi\circ g_n^k-\phi)}{\lambda_n^k}
		\ind{[t_n^k,\,t_{n+1}^k]}(t),
\intertext{whence, expanding  $\phi\circ g_n^k$ in the second term at the right-hand side with Lagrange's reminder (cf. the analogous calculation performed in Section~\ref{sec:limit.mu} for the function $\phi\circ\gamma_n^k$),}
	&= \frac{1}{\Dt_k}\sum_{n=0}^{N_k-1}\left\{\intRd{\phi}{\lambda_{n+1}^k}-
		\intRd{\phi\circ g_n^k}{\lambda_n^k}\right\}\ind{[t_n^k,\,t_{n+1}^k]}(t) \\
	&\phantom{=}+\sum_{n=0}^{N_k-1}\left\{\intRd{v[\lambda_n^k]\cdot\nabla\phi}{\lambda_n^k}+
		\frac{1}{2}\Dt_k\intRd{(D^2\phi(\xlag)v[\lambda_n^k])\cdot v[\lambda_n^k]}{\lambda_n^k}\right\}
			\ind{[t_n^k,\,t_{n+1}^k]}(t),
\end{align*}
where $\xlag$ is a point of the segment connecting $x$ and $x+v[\lambda_n^k]\Dt_k$.

On the other hand, computing as in Section~\ref{sec:limit.mu} we find
\begin{align*}
	\intRd{v[\Lambda_t^k]\cdot\nabla\phi}{\Lambda_t^k} &= 
		\sum_{n=0}^{N_k-1}\intRd{v[\lambda_n^k]\cdot\nabla\phi}{\lambda_n^k}\ind{[t_n^k,\,t_{n+1}^k]}(t) \\
	&\phantom{=}+\sum_{n=0}^{N_k-1}\left\{\frac{t-t_n^k}{\Dt_k}
		\intRd{v[\lambda_n^k]\cdot\nabla\phi}{(\lambda_{n+1}^k-\lambda_n^k)}\right.\\
	&\phantom{=}-\left(\frac{t-t_n^k}{\Dt_k}\right)^2
			\intRd{(v[\lambda_{n+1}^k]-v[\lambda_n^k])\cdot\nabla\phi}{\lambda_n^k} \\
	&\phantom{=}+\left.\left(\frac{t-t_n^k}{\Dt_k}\right)^2
		\intRd{(v[\lambda_{n+1}^k]-v[\lambda_n^k])\cdot\nabla\phi}{\lambda_{n+1}^k}\right\}
			\ind{[t_n^k,\,t_{n+1}^k]}(t),
\end{align*}
hence finally
\begin{align*}
	\frac{d}{dt}\intRd{\phi}{\Lambda_t^k}-\intRd{v[\Lambda_t^k]\cdot\nabla\phi}{\Lambda_t^k} &=
		\frac{1}{\Dt_k}\sum_{n=0}^{N_k-1}\left\{\intRd{\phi}{\lambda_{n+1}^k}-
			\intRd{\phi\circ g_n^k}{\lambda_n^k}\right\}\ind{[t_n^k,\,t_{n+1}^k]}(t) \\
	&\phantom{=}+\sum_{n=0}^{N_k-1}\left\{\frac{1}{2}\Dt_k
		\intRd{(D^2\phi(\xlag)v[\lambda_n^k])\cdot v[\lambda_n^k]}{\lambda_n^k}\right. \\
	&\phantom{=}-\frac{t-t_n^k}{\Dt_k}
		\intRd{v[\lambda_n^k]\cdot\nabla\phi}{(\lambda_{n+1}^k-\lambda_n^k)} \\
	&\phantom{=}+\left(\frac{t-t_n^k}{\Dt_k}\right)^2
			\intRd{(v[\lambda_{n+1}^k]-v[\lambda_n^k])\cdot\nabla\phi}{\lambda_n^k} \\
	&\phantom{=}-\left.\left(\frac{t-t_n^k}{\Dt_k}\right)^2
		\intRd{(v[\lambda_{n+1}^k]-v[\lambda_n^k])\cdot\nabla\phi}{\lambda_{n+1}^k}\right\}
			\ind{[t_n^k,\,t_{n+1}^k]}(t).
\end{align*}
This is formally an equation satisfied by $\Lambda_\ccdot^k$. A derivation analogous to that performed in Section~\ref{sec:limit.mu} gives
\begin{align*}
	\abs{\frac{d}{dt}\intRd{\phi}{\Lambda_t^k}-\intRd{v[\Lambda_t^k]\cdot\nabla\phi}{\Lambda_t^k}} &\leq
		\frac{1}{\Dt_k}\sum_{n=0}^{N_k-1}\abs{\intRd{\phi}{\lambda_{n+1}^k}-
			\intRd{\phi\circ g_n^k}{\lambda_n^k}}\ind{[t_n^k,\,t_{n+1}^k]}(t) \\
	&\phantom{=}+C\sum_{n=0}^{N_k-1}\left\{\Dt_k+\wass{\lambda_n^k}{\lambda_{n+1}^k}\right\}
		\ind{[t_n^k,\,t_{n+1}^k]}(t),
\intertext{where $C>0$ is a constant depending only on $V$, $\lip{v}$, $\lip{\nabla\phi}=\infnorm{D^2\phi}$. In addition, from \eq\eqref{eq:unif.lip.2} we know $\wass{\lambda_n^k}{\lambda_{n+1}^k}\leq(V+2\sqrt{d}\bar{\beta})\Dt_k$, thus}
	&\leq \frac{1}{\Dt_k}\sum_{n=0}^{N_k-1}\abs{\intRd{\phi}{\lambda_{n+1}^k}-
		\intRd{\phi\circ g_n^k}{\lambda_n^k}}\ind{[t_n^k,\,t_{n+1}^k]}(t)+C\Dt_k,
\end{align*}
$C$ being a new constant including the previous one and $V+2\sqrt{d}\bar{\beta}$.

To estimate the remaining term at the right-hand side, we need to adapt Lemma~\ref{lemma:approx.pushfwd} to $g_n^k$.
\begin{lemma} \label{lemma:approx.pushfwd.2}
For all Lipschitz continuous $\varphi:\Rd\to\R$ we have
\begin{equation*}
	\abs{\intRd{\varphi}{\lambda_{n+1}^k}-\intRd{\varphi\circ g_n^k}{\lambda_n^k}}\leq
		C\lip{\varphi}\sqrt{d}h_k
\end{equation*}
where $C:=2+\lip{v}\sup_{k\geq 0}\Dt_k$.
\end{lemma}
\begin{proof}
It suffices to observe that
\begin{align*}
	\abs{\intRd{\varphi(\appr{\gamma}_n^k(x))}{\lambda_n^k(x)}-
		\intRd{\varphi(g_n^k(x))}{\lambda_n^k(x)}} &\leq \lip{\varphi}
			\intRd{\abs{\appr{\gamma}_n^k(x)-g_n^k(x)}}{\lambda_n^k(x)} \\
	&= \Dt_k\lip{\varphi}\sum_{i\in\Zd}
		\intg{E_i^k}{\abs{v[\lambda_n^k](x_i^k)-v[\lambda_n^k](x)}}{\lambda_n^k(x)} \\
	&= \Dt_k\lip{\varphi}\lip{v}\sum_{i\in\Zd}
		\intg{E_i^k}{\abs{x_i^k-x}}{\lambda_n^k(x)} \\
	&\leq \sqrt{d}h_k\Dt_k\lip{\varphi}\lip{v}.
\end{align*}
The thesis follows from this inequality, combining the triangle inequality and Lemma~\ref{lemma:approx.pushfwd}.
\end{proof}

With this result, we can further manipulate the previous inequality and obtain
\begin{equation*}
	\abs{\frac{d}{dt}\intRd{\phi}{\Lambda_t^k}-\intRd{v[\Lambda_t^k]\cdot\nabla\phi}{\Lambda_t^k}}
		\leq C(\beta_k+\Dt_k)
\end{equation*}
which, integrating both sides in time from $0$ to $t\leq T$, implies
\begin{align*}
	\abs{\lint_0^t\left(\frac{d}{d\tau}
		\intRd{\phi}{\Lambda_\tau^k}-\intRd{v[\Lambda_\tau^k]\cdot\nabla\phi}{\Lambda_\tau^k}\right)\,d\tau} &\leq
	\lint_0^t\abs{\frac{d}{d\tau}\intRd{\phi}{\Lambda_\tau^k}-
		\intRd{v[\Lambda_\tau^k]\cdot\nabla\phi}{\Lambda_\tau^k}}\,d\tau \\
	&\leq CT(\beta_k+\Dt_k)
\end{align*}
and finally, taking the limit $k\to\infty$,
\begin{equation}
	\lim_{k\to\infty}\abs{\intRd{\phi}{\Lambda_t^k}-\intRd{\phi}{\lambda_0^k}-
		\lint_0^t\intRd{v[\Lambda_\tau^k]\cdot\nabla\phi}{\Lambda_\tau^k}\,d\tau}=0.
	\label{eq:Lambdatk.limit}
\end{equation}

Thanks to \eq\eqref{eq:Lambdatk.limit}, we are in a position to prove our convergence result of the numerical scheme.
\begin{theorem} \label{theo:convergence}
Let Assumptions~\ref{ass:prop-v}--\ref{ass:mesh.param} hold and assume that the sequence $(\Lambda_\ccdot^k)_{k\geq 0}$ converges to some $\mu_\ccdot$ in $\CTPone$ when $k\to\infty$. Then $\mu_\ccdot$ is a weak solution to problem \eqref{eq:cauchy}.
\end{theorem}
\begin{proof}
Arguing as in Lemma~\ref{lemma:conv.int}, if $\lim_{k\to\infty}\sup_{t\in[0,\,T]}\wass{\Lambda_t^k}{\mu_t}=0$ we know that, up to subsequences,
\begin{equation*}
	\intRd{\phi}{\Lambda_t^k}\to\intRd{\phi}{\mu_t}, \qquad
		\lint_0^t\intRd{v[\Lambda_\tau^k]\cdot\nabla\phi}{\Lambda_\tau^k}\,d\tau\to
			\lint_0^t\intRd{v[\mu_\tau]\cdot\nabla\phi}{\mu_\tau}\,d\tau
				\qquad (k\to\infty).
\end{equation*}
The claim of the theorem then follows from \eq\eqref{eq:Lambdatk.limit} and the arbitrariness of $\phi$, provided we have also $\int_{\Rd}\phi\,d\lambda_0^k\to\int_{\Rd}\phi\,d\incond{\mu}$ when $k\to\infty$. To show the latter fact, it is enough to prove that $\wass{\lambda_0^k}{\incond{\mu}}\to 0$ when $k\to\infty$.

Recall that $d\lambda_0^k=\rho_0^k\,dx$, with $\rho_0^k$ given by \eq\eqref{eq:numscheme.incond}. Fixing $\varphi\in\Lip$ and reasoning like in the proof of Lemma~\ref{lemma:lambda_n_k}-\eqref{lambda_0_k}, we find that
\begin{equation*}
	\intRd{\varphi}{\lambda_0^k}=\intRd{s}{\incond{\mu}},
\end{equation*}
where $s(x)=\sum_{i\in\Zd}\alpha_i\ind{E_i^k}(x)$ and $\alpha_i=\frac{1}{h_k^d}\int_{E_i^k}\varphi(x)\,dx$. Therefore:
\begin{equation*}
	\intRd{\varphi}{(\incond{\mu}-\lambda_0^k)}=\intRd{(\varphi-s)}{\incond{\mu}}=
		\sum_{i\in\Zd}\intg{E_i^k}{(\varphi(x)-\alpha_i)}{\incond{\mu}(x)}.
\end{equation*}
But $\abs{\varphi(x)-\alpha_i}\leq\frac{1}{h_k^d}\int_{E_i^k}\abs{\varphi(x)-\varphi(y)}\,dy\leq\frac{1}{h_k^d}\int_{E_i^k}\abs{x-y}\,dy$, so that from the previous calculations we deduce
\begin{equation*}
	\intRd{\varphi}{(\incond{\mu}-\lambda_0^k)}\leq
		\frac{1}{h_k^d}\sum_{i\in\Zd}\intg{E_i^k}{\intg{E_i^k}{\abs{x-y}}{y}}{\incond{\mu}(x)}\leq\sqrt{d}h_k,
\end{equation*}
whence $\wass{\lambda_0^k}{\incond{\mu}}\leq\sqrt{d}h_k\to 0$ for $k\to\infty$ as desired.
\end{proof}

We conclude this section with a simple criterion which implies the convergence property assumed in Theorem~\ref{theo:convergence}.
\begin{corollary} \label{cor:convergence}
Assume there exists a bounded set $K\subset\Rd$ such that $\supp{\lambda_n^k}\subseteq K$ for all $0\leq n\leq N_k$ and all $k\geq 0$. Then $(\Lambda_\ccdot^k)_{k\geq 0}$ converges in $\CTPone$ to a weak solution of problem \eqref{eq:cauchy} when $k\to\infty$.
\end{corollary}
\begin{proof}
Fix $k\geq 0$, $t\in[0,\,T]$ and let $n$ be such that $t_n^k\leq t\leq t_{n+1}^k$. Let moreover $A\in\B$ be contained in $K^c$, then $\lambda_i^k(A)=0$ for all $i,\,k$ and therefore
\begin{equation*}
	\Lambda_t^k(A)=\left(1-\frac{t-t_n^k}{\Dt_k}\right)\lambda_n^k(A)+\frac{t-t_n^k}{\Dt_k}\lambda_{n+1}^k(A)=0.
\end{equation*}
From the arbitrariness of $k,\,t$ it follows $\supp{\Lambda_t^k}\subseteq K$ for all $k\geq 0$ and all $t\in[0,\,T]$. Since $K$ is bounded, we can find a ball $\ball{R}{0}$ with radius $R>0$ so large that $K\subseteq\ball{R}{0}$. Consequently
\begin{equation*}
	\intRd{\abs{x}^p}{\Lambda_t^k(x)}=\intg{K}{\abs{x}^p}{\Lambda_t^k(x)}\leq R^p<+\infty,
\end{equation*}
\ie, the $\Lambda_t^k$'s have uniformly bounded moments of any order $p\geq 0$. This implies that $\{\Lambda_t^k\}_{k\geq 0}$ is relatively compact in $\Pone$ for all $t\in[0,\,T]$, which, together with the equicontinuity of the family $\{\Lambda_\ccdot^k\}_{k\geq 0}$ entailed by the estimate \eqref{eq:unif.lip.2}, allows us to apply Ascoli-Arzel\`a's Theorem and conclude that the sequence $(\Lambda_\ccdot^k)_{k\geq 0}$ is relatively compact in $\CTPone$. Thus, up to subsequences, it converges to some $\mu_\ccdot\in\CTPone$, which is a weak solution to problem \eqref{eq:cauchy} because of Theorem~\ref{theo:convergence}.
\end{proof}

\begin{remark}[CFL condition]
Let us define $\alpha_k:=V/\beta_k$ for all $k\geq 0$. Then the condition expressed by Assumption~\ref{ass:mesh.param} is equivalent to
\begin{equation}
	V \Dt_k=\alpha_k h_k,
	\label{eq:CFL}
\end{equation}
which implies $\abs{\appr{\gamma}_n^k(x)-x}\leq\alpha_k h_k$ all $x\in\Rd$, \ie, the displacements produced by the mapping $\appr{\gamma}_n^k$ are bounded by the quantity $\alpha_k h_k$. As a consequence, the number of nonempty intersections $\appr{\gamma}_n^k(E_j^k)\cap E_i^k$, $j\in\Zd$, is finite and bounded from above uniformly in $i$, which in particular makes the series in \eq\eqref{eq:numscheme} actually a finite sum for all $i$.

We observe that \eq\eqref{eq:CFL} is a generalization of the Courant-Friedrichs-Lewy (CFL) condition, allowing $\alpha_k\to+\infty$ when $k\to\infty$ to ensure convergence to continuous-in-time-and-space solutions as stated by Theorem~\ref{theo:convergence}. With the time step $\Dt$ frozen, \eq\eqref{eq:CFL} is used in \cite{piccoli2010tem} to prove the stability of the numerical scheme \eqref{eq:numscheme} in approximating the solutions to the \emph{discrete-in-time} model \eqref{eq:pushfwd}.
\end{remark}

\section{Models of crowd and swarm dynamics} \label{sec:models}
In \cite{cristiani2010eai,piccoli2009pfb} a class of models describing the collective dynamics of swarms
and crowds has been introduced. These models are based on the idea that each individual of the group is an \emph{intelligent} (or \emph{active}) agent, able to develop a behavioral strategy for pursuing specific goals. Since agents are not passively dragged by external forces, the Newtonian approach is abandoned in favor of a kinematic one, in which the velocity of the agents stems from few basic behavioral rules. These concepts are formalized in \cite{cristiani2009mso}, where it is shown that the above-cited models can be given a common formulation within the framework of \eq\eqref{eq:cont.eq}. In particular, the corresponding velocity of the agents is
\begin{equation}
	v[\mu_t](x)=\vd(x)+N\intRd{f(\abs{y-x})r(y-x)\smoothind{\neighb_x}(y)}{\mu_t(y)}.
	\label{eq:vel}
\end{equation}

Adopting the terminology of the kinetic theory for active particles, see \cite{bellomo2008mcl}, we call \emph{test agent} an agent potentially concerned with interactions, whose position is described in the above formula by the variable $x$, and \emph{field agents} the agents distributed in space (variable $y$), which the test agent may interact with. In more detail:
\begin{enumerate}
\item $\funct{\vd}{\Rd}{\Rd}$ is the test agent's \emph{desired velocity}, \ie, the velocity of an isolated test agent in the absence of interactions;
\item $\neighb_x$ is the \emph{interaction neighborhood} of the test agent. It conveys the idea that the test agent experiences \emph{nonlocal} interactions with some selected field agents, namely those inside $U_x$. 
\item $\funct{\smoothind{A}}{\Rd}{\R}$ is a cut-off function for the set $A\subseteq\Rd$, such that
\begin{equation*}
	\smoothind{A}(x)=0 \quad \forall\,x\in A^c, \qquad 0\leq\smoothind{A}(x)\leq 1 \quad \forall\,x\in\Rd.
\end{equation*}
In particular, the effect of the term $\smoothind{\neighb_x}$ in \eq\eqref{eq:vel} is to rule out interactions of the test agent with field agents outside $\neighb_x$;
\item $\funct{r}{\Rd}{\Rd}$ is the \emph{direction of the interaction}, depending on the relative position of the interacting agents and oriented in such a way that $r(y-x)\cdot(y-x)\geq 0$, with in addition $\abs{r}\leq 1$;
\item $\funct{f}{\Rplus}{\R}$ is the \emph{interaction strength}, which depends on the distance between the interacting agents. If $f<0$ the interaction is \emph{repulsive}, \ie, the test agent tries to avoid local aggregation with the field agents (as it is common in crowds under normal -- \ie, non-panic -- conditions). If $f>0$ the interaction is \emph{attractive}, \ie, the test agent aims at staying close to the surrounding field agents (as in swarms, where group cohesion is advantageous \eg, for food search or predator avoidance).
\end{enumerate}

The shape of $\neighb_x$ depends partly on the criteria used for selecting the field agents to interact with. In general, when $\neighb_x$ is symmetric \wrt $x$, interactions are said to be \emph{isotropic}, as opposed to \emph{anisotropic} when $\neighb_x$ is not symmetric. Notice also that the interaction integral in \eq\eqref{eq:vel} is actually computed \wrt the mass measure $N\mu_t$, because interactions depend on the number of field agents in $\neighb_x$.

Prototypes of the above objects are, among others: $\neighb_x=\ball{R}{x}$ for isotropic interactions, $\neighb_x=S^\alpha_R(x)$ for anisotropic interactions ($S^\alpha_R(x)$ being the sector of $\ball{R}{x}$ with angular width $\alpha\in(0,\,2\pi]$), $\smoothind{\neighb_x}=\ind{\neighb_x}$,
$f(\abs{y-x})=-1/\abs{y-x}$ for repulsive interactions, $f(\abs{y-x})=\abs{y-x}$ for attractive interactions, $r(y-x)=(y-x)/\abs{y-x}$ (\ie, the unit vector in the direction of $y-x$).
Although meaningful from the modeling point of view, some of these choices need to be adapted in order for the velocity field \eqref{eq:vel} to comply with Assumption~\ref{ass:prop-v}. 

Concerning this, let us introduce the vector-valued function $\funct{F}{\Rd}{\Rd}$, \begin{equation*}
	F(x):=f(\abs{x})r(x),
\end{equation*}
so that \eqref{eq:vel} becomes
\begin{equation}
	v[\mu_t](x)=\vd(x)+N\intRd{F(y-x)\smoothind{\neighb_x}(y)}{\mu_t(y)}.
	\label{eq:vel.F}
\end{equation}
Such a form encompasses a broader class of models of coordinated behavior based on \eq\eqref{eq:cont.eq}, for instance those in \cite{canuto2008eaa}. From now on, we will take \eqref{eq:vel.F} as our velocity model.

\begin{assumption}[Properties of $v$ as in \eq\eqref{eq:vel.F}] \label{ass:prop-v.model}
We assume that the velocity field in \eq\eqref{eq:vel.F} satisfies the following properties.
\begin{enumerate}
\item \label{ass:bound-vd} $\vd$ is Lipschitz continuous and bounded in $\Rd$.
\item \label{ass:isometry} There exists $R>0$ such that, for each $x\in\Rd$, the set $\neighb_x\subseteq\ball{R}{x}$ is measurable and isometric to a reference set $\neighb_0\subseteq\ball{R}{0}$.  
\item \label{ass:F} $F$ is Lipschitz continuous in $\ball{R}{0}$.
\item \label{ass:smthind} $\smoothind{A}$ is Lipschitz continuous in $\Rd$ for each measurable $A\subseteq\Rd$.
\end{enumerate}
\end{assumption}

\begin{remark}
Some comments on Assumption~\ref{ass:prop-v.model} are in order.
\begin{enumerate}
\item In the sequel we denote
\begin{equation*}
	\Vd:=\sup_{x\in\Rd}\abs{\vd(x)}<+\infty.
\end{equation*}
\item Saying that $\neighb_x$ is isometric to $\neighb_0$ means that there is a rigid transformation $\funct{\xi_x}{\Rd}{\Rd}$ mapping $\neighb_0$ onto $\neighb_x$: $\xi_x(\neighb_0)=\neighb_x$.
Specifically, $\xi_x$ has the form
\begin{equation}
	\xi_x(z)=\rot_x z+x,
	\label{eq:isometry}
\end{equation}
where $\rot_x\in\R^{d\times d}$ is a rotation matrix possibly depending on the point $x$.
\item Lipschitz continuity of $F$ is required {\em only} in the ball $\ball{R}{0}$, not in the whole space, \ie, the condition
$	\abs{F(z_2)-F(z_1)}\leq\lip{F}\abs{z_2-z_1}$
must hold only for $\abs{z_1},\,\abs{z_2}\leq R$. This implies that $F$ is bounded in $\ball{R}{0}$, that is,
$	\abs{F(z)}\leq\Fmax \quad \text{for\ } \abs{z}\leq R,$
with $\Fmax:=\max\{\abs{F(0)},\,\lip{F}\}(1+R).$
\item The function $\smoothind{A}$ can be thought of as a mollification of $\ind{A}$. In particular, by continuity $\smoothind{A}=0$ on $\partial A$. Because of the isometry between $\neighb_0$ and $\neighb_x$, one can check that the following relation holds:
\begin{equation}
	\smoothind{\neighb_x}(y)=\smoothind{\neighb_0}(\xi_x^{-1}(y)), \quad \forall\,y\in\Rd.
	\label{eq:trans.smoothind}
\end{equation}
\end{enumerate}
\end{remark}

In the next two sections we show that, with both isotropic and anisotropic interactions, Assumption~\ref{ass:prop-v.model} is sufficient to apply our existence and approximation results to the above-mentioned crowd and swarm models. First, we consider the case of a spherical neighborhood, which is the most important prototype for isotropic interactions. Later, we extend the analysis to any bounded neighborhood, including the significant cases of anisotropic interactions.

\subsection{Spherical interaction neighborhood} \label{sec:isotropic}
In case of isotropic interactions we set $\neighb_x=\ball{R}{x}$, and consequently $\neighb_0=\ball{R}{0}$. Since $\ball{R}{x}$ is invariant under rotations, the mapping $\xi_x$ is simply $\xi_x(z)=z+x$.

By the change of variables $z=\xi_x^{-1}(y)$ in \eq\eqref{eq:vel.F}, and recalling furthermore \eq\eqref{eq:trans.smoothind}, we can rewrite the velocity as
\begin{equation*}
	v[\mu_t](x)=\vd(x)+N \intRd{F(z)\smoothind{\ball{R}{0}}(z)}{(\xi_x^{-1}\#\mu_t)(z)},
\end{equation*}
where $\xi_x^{-1}\#\mu_t$ is in $\Pone$ for all $\mu_t\in\Pone$ and all fixed $x$.

We state a preliminary result, which is useful to verify that such a velocity field complies with Assumption~\ref{ass:prop-v}.
\begin{lemma} \label{lemma:lip.F.smoothind}
The function $z \mapsto F(z)\smoothind{\ball{R}{0}}(z)$ is Lipschitz continuous in $\Rd$.
\end{lemma}
\begin{proof}
To study the expression
\begin{equation*}
	e(z_1,\,z_2):=\abs{F(z_2)\smoothind{\ball{R}{0}}(z_2)-F(z_1)\smoothind{\ball{R}{0}}(z_1)}, \quad
		z_1,\,z_2\in\Rd
\end{equation*}
it is convenient to distinguish three cases.
\begin{enumerate}
\item If $z_1,\,z_2\in\ball{R}{0}$ we have
\begin{align*}
	e(z_1,\,z_2) &\leq \abs{F(z_2)\smoothind{\ball{R}{0}}(z_2)-F(z_2)\smoothind{\ball{R}{0}}(z_1)}+
		\abs{F(z_2)\smoothind{\ball{R}{0}}(z_1)-F(z_1)\smoothind{\ball{R}{0}}(z_1)} \\
	&\leq (\Fmax\lip{\smoothind{\ball{R}{0}}}+\lip{F})\abs{z_2-z_1}.
\end{align*}
\item If $z_1\not\in\ball{R}{0}$ and $z_2\in\ball{R}{0}$ (or vice versa) then $e(z_1,\,z_2)=\abs{F(z_2)\smoothind{\ball{R}{0}}(z_2)}$. Let $z_\theta:=\theta z_1+(1-\theta)z_2$, $\theta\in[0,\,1]$, be a point of the segment connecting $z_1$ to $z_2$ and pick $\bar{\theta}$ such that $z_{\bar{\theta}}\in\partial\ball{R}{0}$. Since $\smoothind{\ball{R}{0}}(z_{\bar{\theta}})=0$, we have
\begin{equation*}
	e(z_1,\,z_2)=\abs{F(z_2)\smoothind{\ball{R}{0}}(z_2)-
		F(z_{\bar{\theta}})\smoothind{\ball{R}{0}}(z_{\bar{\theta}})}\leq
			(\Fmax\lip{\smoothind{\ball{R}{0}}}+\lip{F})\abs{z_2-z_{\bar{\theta}}}.
\end{equation*}
On the other hand, $\abs{z_2-z_{\bar{\theta}}}=\bar{\theta}\abs{z_2-z_1}\leq\abs{z_2-z_1}$.
\item If $z_1,\,z_2\not\in\ball{R}{0}$ then $e(z_1,\,z_2)=0\leq\abs{z_2-z_1}$. \qedhere
\end{enumerate}
\end{proof}

\begin{proposition}[Velocity with isotropic interactions] \label{prop:isotropic}
Let Assumption~\ref{ass:prop-v.model} hold with $\neighb_0=\ball{R}{0}$. Then the velocity field \eqref{eq:vel.F} complies with Assumption~\ref{ass:prop-v}.
\end{proposition}
\begin{proof}
In the sequel, $x\in\Rd$ and $\mu\in\Pone$ are fixed but arbitrary. We verify the items of Assumption~\ref{ass:prop-v} in order.
\begin{enumerate}
\item The field $v$ is uniformly bounded thanks to
\begin{equation*}
	\abs{v[\mu](x)}\leq\Vd+N\intRd{\abs{F(z)\smoothind{\ball{R}{0}}(z)}}{(\xi_x^{-1}\#\mu)(z)}
		\leq\Vd+N\Fmax\mu(\xi_x(\Rd))=\Vd+N\Fmax.
\end{equation*}
\item As for Lipschitz continuity, we check it separately \wrt to $x$ and to $\mu$.
\begin{enumerate}
\item Let $x_1,\,x_2\in\Rd$, then, recalling Lemma~\ref{lemma:lip.F.smoothind},
\begin{align*}
	\abs{v[\mu](x_2)-v[\mu](x_1)} &\leq \abs{\vd(x_2)-\vd(x_1)}+
		N\abs{\intRd{F(z)\smoothind{\ball{R}{0}}(z)}{(\xi_{x_2}^{-1}\#\mu-\xi_{x_1}^{-1}\#\mu)(z)}} \\
	&\leq\lip{\vd}\abs{x_2-x_1}+N\lip{F\smoothind{\ball{R}{0}}}
		\wass{\xi_{x_1}^{-1}\#\mu}{\xi_{x_2}^{-1}\#\mu}.
\end{align*}
In addition,
\begin{align*}
	\wass{\xi_{x_1}^{-1}\#\mu}{\xi_{x_2}^{-1}\#\mu} &= \sup_{\varphi\in\Lip}
		\intRd{(\varphi\circ\xi_{x_2}^{-1}-\varphi\circ\xi_{x_1}^{-1})}{\mu} \\
	&\leq \intRd{\abs{\xi_{x_2}^{-1}(y)-\xi_{x_1}^{-1}(y)}}{\mu(y)}=\abs{x_2-x_1},
\end{align*}
whence the Lipschitz continuity of $x\mapsto v[\mu](x)$.
\item Let now $\mu,\,\nu\in\Pone$, then, invoking again Lemma~\ref{lemma:lip.F.smoothind},
\begin{align*}
	\abs{v[\nu](x)-v[\mu](x)} &= N\abs{
		\intRd{F(z)\smoothind{\ball{R}{0}}(z)}{(\xi_x^{-1}\#\nu-\xi_x^{-1}\#\mu)(z)}} \\
	&\leq N\lip{F\smoothind{\ball{R}{0}}}\wass{\xi_x^{-1}\#\mu}{\xi_x^{-1}\#\nu} \\
	&= N\lip{F\smoothind{\ball{R}{0}}}\sup_{\varphi\in\Lip}\intRd{\varphi\circ\xi_x^{-1}}{(\nu-\mu)} \\
	&= N\lip{F\smoothind{\ball{R}{0}}}\wass{\mu}{\nu},
\end{align*}
since $\xi_x^{-1}\in\Lip$ implies that $\varphi\circ\xi_x^{-1}$ spans the whole space $\Lip$ when $\varphi$ varies in $\Lip$. Hence also $\mu\mapsto v[\mu](x)$ is Lipschitz continuous.
\end{enumerate}
\item Finally we examine the linearity \wrt to the probability for convex combinations. Let $\mu,\,\nu\in\Pone$ and $\alpha\in [0,\,1]$, then
\begin{align*}
	v[\alpha\mu+(1-\alpha)\nu](x) &= \vd(x)+\alpha N\intRd{F(z)\smoothind{\ball{R}{0}}(z)}{(\xi_x^{-1}\#\mu)(z)} \\
	&\phantom{=}+(1-\alpha)N\intRd{F(z)\smoothind{\ball{R}{0}}(z)}{(\xi_x^{-1}\#\nu)(z)}.
\end{align*}
Writing $\vd(x)=\alpha\vd(x)+(1-\alpha)\vd(x)$ and collecting the coefficients $\alpha$ and $1-\alpha$ gives the result. \qedhere
\end{enumerate}
\end{proof}

Owing to Proposition~\ref{prop:isotropic}, we can state that models based on \eq\eqref{eq:cont.eq}, with a velocity field featuring isotropic interactions as in \eq\eqref{eq:vel.F}, see \eg, \cite{piccoli2009pfb}, admit probability measure solutions for any initial distribution of the agents with finite first and second order moments (\eg, an initial distribution with compact support).  Moreover, such solutions can be duly approximated using the numerical scheme \eqref{eq:numscheme}.

\begin{remark}[Unbounded interaction neighborhood]
As a modification to Assumption~\ref{ass:prop-v.model}, we can allow $\neighb_0=U_x=\Rd$ along with the boundedness assumption $F(x)\leq\Fmax$ for all $x\in\Rd$. The above arguments can be promptly adapted to show that also in this case the velocity field \eqref{eq:vel.F} complies with Assumption~\ref{ass:prop-v}. However, it should be noted that, for the applications we have in mind, the physically relevant cases are those in which the interaction neighborhood is bounded.
\end{remark}

\subsection{Bounded interaction neighborhood} \label{sec:anisotropic}
In this section we drop the specific hypothesis that $\neighb_x$ be a ball. We allow it to have a generic shape, with the only requirement of being bounded, \ie, contained in a ball. This encompasses the important case of anisotropic interactions. In such cases, the neighborhood $\neighb_x$ may not be invariant under rotations, therefore we have to consider the full form \eqref{eq:isometry} of the transformation $\xi_x$. Performing again the change of variables $z=\xi_x^{-1}(y)$ in the integral \eqref{eq:vel.F}, the velocity takes now the form
\begin{equation*}
	v[\mu_t](x)=\vd(x)+N\intRd{F(\rot_x z)\smoothind{\neighb_0}(z)}{(\xi_x^{-1}\#\mu_t)(z)}
\end{equation*}
with the rotation matrix explicitly appearing in the argument of the function $F$. To deal with it, for the sake of simplicity we confine ourselves to the two-dimensional case ($d=2$), for then a simple representation of $\rot_x$ is available:
\begin{equation*}
	\rot_x=
	\begin{pmatrix}
		\cos{\rotang_x} & -\sin{\rotang_x} \\
		\sin{\rotang_x} & \cos{\rotang_x}
	\end{pmatrix},
\end{equation*}
$\rotang_x\in[0,\,2\pi)$ being the angle of rotation which determines the local orientation in the point $x$ of the neighborhood of interaction.

The choice of $\rotang_x$ has to do with the way in which the anisotropy of the interactions is modeled. In the models of crowd and swarm dynamics we are considering, $\rotang_x$ is the angle formed by the desired velocity $\vd(x)$ \wrt a fixed reference direction, for instance the horizontal one identified by the unit vector $\ivect$. Assuming for simplicity that $\vd$ has constant unit modulus, this implies
\begin{equation}
	\cos{\rotang_x}=\vd(x)\cdot\ivect, \qquad \sin{\rotang_x}=(\vd(x)\times\ivect)\cdot\kvect.
	\label{eq:cos.sin}
\end{equation}
In the second formula, $\vd(x)$ and $\ivect$ are thought of as embedded into $\R^3$, with $\times$ denoting vector product and $\kvect$ the unit vector orthogonal to the plane of $\vd(x)$ and $\ivect$.

\begin{remark}
More in general, \eq\eqref{eq:cos.sin} holds with $\vd(x)$ replaced by $\vd(x)/\abs{\vd(x)}$, which is a Lipschitz continuous field provided $\abs{\vd}$ is uniformly bounded away from zero.
\end{remark}

The formalism introduced above allows us to prove the following technical fact.
\begin{lemma} \label{lemma:lip.rot}
For all $x_1,\,x_2,\,z\in\R^2$ it results
\begin{equation*}
	\abs{(\rot_{x_2}-\rot_{x_1})z}\leq\sqrt{2}\lip{\vd}\abs{x_2-x_1}\abs{z}.
\end{equation*}
The same holds with $\rot_{x_1}$, $\rot_{x_2}$ replaced by $\rot_{x_1}^{-1}$, $\rot_{x_2}^{-1}$.
\begin{proof}
A straightforward computation shows that
\begin{equation*}
	\abs{(\rot_{x_2}-\rot_{x_1})z}=\sqrt{(\cos{\rotang_{x_2}}-\cos{\rotang_{x_1}})^2+
		(\sin{\rotang_{x_2}}-\sin{\rotang_{x_1}})^2}\abs{z},
\end{equation*}
and the same is true also using inverse matrices. In addition,
\begin{eqnarray*}
	\abs{\cos{\rotang_{x_2}}-\cos{\rotang_{x_1}}} & = & \abs{(\vd(x_2)-\vd(x_1))\cdot\ivect}\leq
		\abs{\vd(x_2)-\vd(x_1)}, \\
	\abs{\sin{\rotang_{x_2}}-\sin{\rotang_{x_1}}} & = & \abs{[(\vd(x_2)-\vd(x_1))\times\ivect]\cdot\kvect}\leq
		\abs{\vd(x_2)-\vd(x_1)},
\end{eqnarray*}
hence the thesis follows from the Lipschitz continuity of $\vd$.
\end{proof}
\end{lemma}

With Lemma~\ref{lemma:lip.rot} we are in a position to prove that the velocity \eqref{eq:vel.F} complies with Assumption~\ref{ass:prop-v} also in case of anisotropic interactions.
\begin{proposition}[Velocity with anisotropic interactions] \label{prop:anisotropic}
Let Assumption~\ref{ass:prop-v.model} hold and assume $d=2$. Then the velocity field \eqref{eq:vel.F} complies with Assumption~\ref{ass:prop-v}.
\end{proposition}
\begin{proof}
\begin{enumerate}
\item Uniform boundedness, Lipschitz continuity \wrt the probability, and linearity \wrt to the probability for convex combinations follow straightforwardly from calculations entirely analogous to those performed in the proof of Proposition~\ref{prop:isotropic}. In fact, it is sufficient to observe that, $\rot_x$ being an isometry, the function $F(\rot_x\cdot)$ is Lipschitz continuous and bounded in $\ball{R}{0}$ with $\lip{F(\rot_x\cdot)}=\lip{F}$ for all $x\in\Rd$. Moreover, by the same argument as in the proof of Lemma~\ref{lemma:lip.F.smoothind}, the function $F(\rot_x\cdot)\smoothind{\neighb_0}(\cdot)$ is Lipschitz continuous in $\Rd$ with the same Lipschitz constant as $F\smoothind{\neighb_0}$, thus in particular independent of $x$.

\item Lipschitz continuity \wrt to $x$ is instead more delicate, because it involves directly the rotation matrix $\rot_x$. Let $x_1,\,x_2\in\R^2$ and fix $\mu\in\PP_1(\R^2)$, then
\begin{align}
	\abs{v[\mu](x_2)-v[\mu](x_1)} &\leq \abs{\vd(x_2)-\vd(x_1)} \notag \\
	&\phantom{\leq} +N\abs{\intRd{F(\rot_{x_2}z)\smoothind{\neighb_0}(z)}{(\xi_{x_2}^{-1}\#\mu)(z)}
		-\intRd{F(\rot_{x_1}z)\smoothind{\neighb_0}(z)}{(\xi_{x_1}^{-1}\#\mu)(z)}} \notag \\
	&\leq\lip{\vd}\abs{x_2-x_1}+
		N\intRd{\abs{F(\rot_{x_2}z)-F(\rot_{x_1}z)}\smoothind{\neighb_0}(z)}{(\xi_{x_2}^{-1}\#\mu)(z)} \notag \\
	&\phantom{\leq}+N\abs{
		\intRd{F(\rot_{x_1}z)\smoothind{\neighb_0}(z)}{(\xi_{x_2}^{-1}\#\mu-\xi_{x_1}^{-1}\#\mu)(z)}}.
			\label{eq:rhs-integrals}
\end{align}
In the first integral at the right-hand side of \eqref{eq:rhs-integrals} we can assume $z\in\neighb_0$, for otherwise $\smoothind{\neighb_0}(z)=0$. Hence $\abs{z}\leq R$ and moreover $\rot_x z\in\ball{R}{0}$ for all $x\in\R^2$. Lipschitz continuity of $F$ in that ball, along with Lemma~\ref{lemma:lip.rot}, implies
\begin{equation*}
	\abs{F(\rot_{x_2}z)-F(\rot_{x_1}z)}\leq\lip{F}\abs{(\rot_{x_1}-\rot_{x_2})z}\leq
		\sqrt{2}\lip{F}\lip{\vd}R\abs{x_2-x_1},
\end{equation*}
so that finally
\begin{align}
	N\intRd{\abs{F(\rot_{x_2}z)-F(\rot_{x_1}z)}\smoothind{\neighb_0}(z)}{(\xi_{x_2}^{-1}\#\mu)(z)} &\leq
		N\intg{\neighb_0}{\abs{F(\rot_{x_2}z)-F(\rot_{x_1}z)}}{(\xi_{x_2}^{-1}\#\mu)(z)} \notag \\
	&\leq N\sqrt{2}\lip{F}\lip{\vd}R\abs{x_2-x_1},
		\label{eq:anisotropic.rhs.1int}
\end{align}
having observed that $(\xi_{x_2}^{-1}\#\mu)(\neighb_0)=\mu(\neighb_{x_2})\leq 1$.

As far as the second integral at the right-hand side of \eqref{eq:rhs-integrals} is concerned, we have
\begin{align}
	&N\abs{\intRd{F(\rot_{x_1}z)\smoothind{\neighb_0}(z)}{(\xi_{x_2}^{-1}\#\mu-\xi_{x_1}^{-1}\#\mu)(z)}} \notag \\
	& \quad \leq N\intRd{\abs{F(\rot_{x_1}\xi_{x_2}^{-1}(y))\smoothind{\neighb_0}(\xi_{x_2}^{-1}(y))
		-F(\rot_{x_1}\xi_{x_1}^{-1}(y))\smoothind{\neighb_0}(\xi_{x_1}^{-1}(y))}}{\mu(y)} \notag \\
	& \quad = N\intg{\neighb_{x_1}\cup\neighb_{x_2}}
		{\abs{F(\rot_{x_1}\xi_{x_2}^{-1}(y))\smoothind{\neighb_0}(\xi_{x_2}^{-1}(y))
			-F(\rot_{x_1}\xi_{x_1}^{-1}(y))\smoothind{\neighb_0}(\xi_{x_1}^{-1}(y))}}{\mu(y)}.
			\label{eq:anisotropic.rhs.2int}
\end{align}
Notice that we can confine ourselves to $y\in\neighb_{x_1}\cup\neighb_{x_2}$, for otherwise $\smoothind{\neighb_0}(\xi_{x_j}^{-1}(y))=\smoothind{\neighb_{x_j}}(y)=0$ for both $j=1,\,2$. We distinguish two cases.
\begin{enumerate}
\item $\abs{x_2-x_1}>2R$. \\ In this case $\neighb_{x_1}\cap\neighb_{x_2}=\emptyset$ because the balls $\ball{R}{x_1}$, $\ball{R}{x_2}$ are disjoint. Thus:
\begin{align*}
	\eqref{eq:anisotropic.rhs.2int} &=
		N\intg{\neighb_{x_1}}{\abs{F(\rot_{x_1}\xi_{x_1}^{-1}(y))}\smoothind{\neighb_{x_1}}(y)}{\mu(y)}+
		N\intg{\neighb_{x_2}}{\abs{F(\rot_{x_1}\xi_{x_2}^{-1}(y))}\smoothind{\neighb_{x_2}}(y)}{\mu(y)}.
\intertext{For all $y\in\neighb_{x_j}$, $j=1,\,2$, it results $\xi_{x_j}^{-1}(y)\in\neighb_0\subset\ball{R}{0}$, hence $\rot_{x_1}\xi_{x_j}^{-1}(y)\in\ball{R}{0}$ and we can use the boundedness of $F$ in that ball to get}
	&\leq N\Fmax\left(\intg{\neighb_{x_1}}{\smoothind{\neighb_{x_1}}(y)}{\mu(y)}+
		\intg{\neighb_{x_2}}{\smoothind{\neighb_{x_2}}(y)}{\mu(y)}\right)\leq
			N\Fmax\mu(\neighb_{x_1}\cup\neighb_{x_2})\leq N\Fmax.
\end{align*}
But $1<\frac{\abs{x_2-x_1}}{2R}$, therefore we conclude
\begin{equation}
	\eqref{eq:anisotropic.rhs.2int}\leq\frac{N\Fmax}{2R}\abs{x_2-x_1}.
	\label{eq:anisotropic.rhs.2int.1case}
\end{equation}
\item $\abs{x_2-x_1}\leq 2R$. \\ In this case the neighborhoods $\neighb_{x_1}$, $\neighb_{x_2}$ need not be disjoint, however we can resort to the Lipschitz continuity of the function $F(\rot_x\cdot)\smoothind{\neighb_0}(\cdot)$:
\begin{align*}
	\eqref{eq:anisotropic.rhs.2int} &\leq
		\lip{F\smoothind{\neighb_0}}N
			\intg{\neighb_{x_1}\cup\neighb_{x_2}}{\abs{\xi_{x_2}^{-1}(y)-\xi_{x_1}^{-1}(y)}}{\mu(y)} \\
	&= \lip{F\smoothind{\neighb_0}}N
			\intg{\neighb_{x_1}\cup\neighb_{x_2}}{\abs{\rot_{x_2}^{-1}(y-x_2)-\rot_{x_1}^{-1}(y-x_1)}}{\mu(y)} \\
	&= \lip{F\smoothind{\neighb_0}}N
			\intg{\neighb_{x_1}\cup\neighb_{x_2}}{\abs{(\rot_{x_2}^{-1}-\rot_{x_1}^{-1})(y-x_1)
				+\rot_{x_2}^{-1}(x_2-x_1)}}{\mu(y)}
\intertext{and further, thanks to Lemma~\ref{lemma:lip.rot} and to the fact that $\rot_{x_2}^{-1}$ is an isometry,}
	&\leq \lip{F\smoothind{\neighb_0}}N\left(
			\sqrt{2}\lip{\vd}\intg{\neighb_{x_1}\cup\neighb_{x_2}}{\abs{y-x_1}}{\mu(y)}
				+\mu(\neighb_{x_1}\cup\neighb_{x_2})\right)\abs{x_2-x_1}.
\end{align*}
Let us examine the term with the integral. If $y\in\neighb_{x_1}$ then $\abs{y-x_1}\leq R$ whereas if $y\in\neighb_{x_2}$ then $\abs{y-x_1}\leq\abs{y-x_2}+\abs{x_2-x_1}\leq 3R$. Finally, $\abs{y-x_1}\leq 3R$ for all $y\in\neighb_{x_1}\cup\neighb_{x_2}$, which says
\begin{equation}
	\eqref{eq:anisotropic.rhs.2int}\leq
		\lip{F\smoothind{\neighb_0}}N(3\sqrt{2}\lip{\vd}R+1)\abs{x_2-x_1}.
	\label{eq:anisotropic.rhs.2int.2case}		
\end{equation}
\end{enumerate}

In conclusion, from \eqref{eq:anisotropic.rhs.2int.1case} and \eqref{eq:anisotropic.rhs.2int.2case} we deduce that there exists a constant $C>0$ such that $\eqref{eq:anisotropic.rhs.2int}\leq C\abs{x_2-x_1}$ for all $x_1,\,x_2\in\R^2$. This, together with the estimate \eqref{eq:anisotropic.rhs.1int}, completes the proof of Lipschitz continuity of the mapping $x\mapsto v[\mu](x)$. \qedhere
\end{enumerate}
\end{proof}

In view of Proposition~\ref{prop:anisotropic} we conclude that two-dimensional models based on the velocity \eqref{eq:vel.F} with anisotropic interactions have probability measure solutions, which can be approximated arbitrarily well using the scheme \eqref{eq:numscheme} on finer and finer numerical grids. Notice that, as far as \eg, crowd dynamics is concerned, two-dimensional problems are enough for applications.

\begin{remark}[Higher dimension]
For $d>2$ additional technicalities arise, due to a more complex structure of the rotation matrix. Nevertheless, in the special case that the desired velocity is constant in $x$, it is straightforward to extend the results to any spatial dimension. In fact, the rotation matrix being independent of $x$, Proposition~\ref{prop:anisotropic} can be proved without using Lemma~\ref{lemma:lip.rot}, which is the only point where we use the explicit representation of the matrix. Models with constant desired velocity have been recently proposed for swarm dynamics problems \cite{cristiani2010eai} and for \emph{rendez-vous} algorithms \cite{canuto2008eaa}.
\end{remark}

\begin{remark}[Zero desired velocity]
When interactions are anisotropic and the desired velocity is zero \cite{cristiani2010eai}, the orientation of the neighborhood of interaction cannot be defined by \eq\eqref{eq:cos.sin}. However, this issue can be solved by replacing $\vd$ in \eqref{eq:cos.sin} with any other Lipschitz continuous unit vector field, \eg, a nonzero constant one, with the only purpose of defining a rotation angle. Clearly, this problem does not arise if the desired velocity is zero but interactions are isotropic, as in \cite{canuto2008eaa}.
\end{remark}

\section{Case study: discrete models} \label{sec:case.study}
In this last section we put the theory into operation by giving an example of explicit solution to the Cauchy problem \eqref{eq:cauchy}. Furthermore, we visualize the convergence to such solution of the approximations produced by the numerical scheme discussed in Section~\ref{sec:approx}.

For illustrative purposes, we consider the simple case of a system of agents featuring isotropic interactions with zero desired velocity:
\begin{equation}
	v[\mu_t](x)=N\intRd{F(y-x)\smoothind{\ball{R}{x}}(y)}{\mu_t(y)}.
	\label{eq:case-study.vel}
\end{equation}
Moreover, we prescribe as initial condition the discrete probability measure
\begin{equation}
	\incond{\mu}=\frac{1}{N}\sum_{l=1}^{N}\delta_{x^l_0},
	\label{eq:case-study.incond}
\end{equation}
where $\delta_{x^l_0}$ is the Dirac mass concentrated in $x^l_0$ (\ie, for any $A\in\B$ it results $\delta_{x^l_0}(A)=1$ if $x^l_0\in A$, $\delta_{x^l_0}(A)=0$ otherwise) and $x^1_0,\,\dots,\,x^N_0$ are $N$ selected points in $\Rd$.

We recall that $\incond{\mu}$ is the common law of the random variables $X^i_0$, $i=1,\,\dots,\,N$, expressing the initial positions of the agents. The structure \eqref{eq:case-study.incond} of $\incond{\mu}$ implies that each $X^i_0$ is a random variable taking almost surely the values $x^1_0,\,\dots,\,x^N_0$, each with probability $1/N$. Therefore we are considering a discrete model of a group of indistinguishable agents initially concentrated in $x^1_0,\,\dots,\,x^N_0$. The indistinguishability is reflected by the fact that any of the points $x^l_0$ can be, with equal probability, the initial position of the generic $i$-th agent.

We find a solution to the Cauchy problem \eqref{eq:cauchy} with initial condition \eqref{eq:case-study.incond} by the method of the characteristics (cf. Section~\ref{sec:cauchy}). In particular, we know from \eq\eqref{eq:rep.formula} that $\mu_t=\gamma_t\#\incond{\mu}$, where $\gamma_t$ is the flow map. From the linearity of the push forward we first deduce $\mu_t=\frac{1}{N}\sum_{l=1}^{N}\gamma_t\#\delta_{x^l_0}$, then we observe that for any measurable set $A$ it results
\begin{equation*}
	(\gamma_t\#\delta_{x^l_0})(A)=\delta_{x^l_0}(\gamma_t^{-1}(A))=
		\begin{cases}
			1 & \text{if\ } x^l_0\in\gamma_t^{-1}(A) \Leftrightarrow \gamma_t(x^l_0)\in A \\
			0 & \text{otherwise.}
		\end{cases}
\end{equation*}
Hence $\gamma_t\#\delta_{x^l_0}=\delta_{\gamma_t(x^l_0)}$ and we can write the solution as
\begin{equation}
	\mu_t=\frac{1}{N}\sum_{l=1}^{N}\delta_{\gamma_t(x^l_0)}.
	\label{eq:case-study.solution}
\end{equation}

Recalling that $\mu_t$ is the  law of the random variables $\condexp{X}^i_t=\Exp{X^i_t\vert X^i_0}$, $i=1,\,\dots,\,N$, from \eq\eqref{eq:case-study.solution} we infer that each $\condexp{X}^i_t$ takes almost surely the values $\gamma_t(x^1_0),\,\dots,\,\gamma_t(x^N_0)$, each with probability $1/N$. Therefore, at every time $t>0$ the distribution of the group of agents is concentrated on the discrete set of points $\gamma_t(x^1_0),\,\dots,\,\gamma_t(x^N_0)$. Notice that again we cannot associate deterministically a given agent with its position because of the indistinguishability of the agents. However, we can describe the trajectories of the agents by means of the mappings $t\mapsto\gamma_t(x^l_0)$, $l=1,\,\dots,\,N$.

The flow map is defined by \eq\eqref{eq:gammat}, which with the velocity \eqref{eq:case-study.vel} and the initial condition \eqref{eq:case-study.incond} yields
\begin{equation}
	\begin{cases}
		\dfrac{\partial\gamma_t(x)}{\partial t}=
			\ds\sum_{j=1}^{N}F(\gamma_t(x^j_0)-\gamma_t(x))\smoothind{\ball{R}{\gamma_t(x)}}(\gamma_t(x^j_0)) \\
			\gamma_0(x)=x
	\end{cases}
	\label{eq:case-study.gamma}
\end{equation}
for all $x\in\Rd$. The discrete structure \eqref{eq:case-study.solution} of $\mu_t$ makes it actually sufficient to solve problem \eqref{eq:case-study.gamma} for $x=x^l_0$. Setting $x_l(t):=\gamma_t(x^l_0)$ and computing \eq\eqref{eq:case-study.gamma} for $x=x^l_0$ we find the initial-value problem
\begin{equation}
	\begin{cases}
		\dot{x}_l=\ds\sum_{j=1}^{N}F(x_j-x_l)\smoothind{\ball{R}{x_l}}(x_j) \\[4mm]
		x_l(0)=x^l_0
	\end{cases}
	\qquad (l=1,\,\dots,\,N),
	\label{eq:case-study.ode}
\end{equation}
thus we conclude that constructing the measure \eqref{eq:case-study.solution} amounts to solving the system of ODEs \eqref{eq:case-study.ode}, whose solutions are the positions of the agents at $t>0$.

From the numerical point of view, we can either approximate the measure \eqref{eq:case-study.solution} by using the scheme \eqref{eq:numscheme} or integrate problem \eqref{eq:case-study.ode} via a standard numerical method for ODEs. The remaining part of this section is devoted to show the convergence of the approximations obtained from \eqref{eq:numscheme} to the numerical solutions of \eqref{eq:case-study.ode} in a toy model. Let us consider a one-dimensional problem ($d=1$) with $N=10$ agents, whose initial positions are sampled from a uniform distribution on the interval $[0,\,1]$. Agents repel each other according to the following repulsion function:
\begin{equation*}
	F(z)=-\frac{az}{\max^2\{\abs{z},\,\eps\}} \qquad (a,\,\eps>0)
\end{equation*}
which is from the product of
\begin{equation*}
	f(z)=-\frac{a}{\max\{\abs{z},\,\eps\}}, \qquad r(z)=\frac{z}{\max\{\abs{z},\,\eps\}}.
\end{equation*}
The repulsion strength $f$ is inversely proportional to the distance between the interacting agents (up to a minimal threshold $\eps$) and the direction of the interaction $r$ is a Lipschitz mollification of the unit vector $z/\abs{z}$. The reference interaction neighborhood is $\ball{R}{0}=(-R,\,R)$ with cut-off function
\begin{equation*}
	\smoothind{\ball{R}{0}}(z)=e^{-\frac{b\abs{z}^2}{R^2-\abs{z}^2}}\ind{\ball{R}{0}}(z) \qquad (b>0),
\end{equation*}
which is a $C^\infty$ mollification of the indicator function of $\ball{R}{0}$. Parameters are set to $R=0.1$, $a=0.01$, $b=0.02$, $\eps=R/4$. The computational time is $T=0.1$.

Simulations of the ODE system \eqref{eq:case-study.ode} and of problem \eqref{eq:cauchy} with initial condition \eqref{eq:case-study.incond} were run independently and their results visualized on the same graphs of Fig.~\ref{fig:simulation} for duly comparison. In particular, the ODE system was numerically integrated using an explicit-in-time Euler scheme, whereas the conservation law for the probability $\mu_t$ was solved through the scheme \eqref{eq:numscheme} on different meshes, choosing
\begin{equation*}
	\Dt_k=\left(\frac{h_k}{V}\right)^\delta \qquad (0<\delta<1).
\end{equation*}
Note that this entails $\beta_k\sim h_k^{1-\delta}$ in Assumption~\ref{ass:mesh.param} and $\alpha_k\sim h_k^{\delta-1}$ in \eq\eqref{eq:CFL}. Figure~\ref{fig:simulation} displays the numerical solution computed with $\delta=0.9$ and $h_k=1/k$ in the three cases $k=10^2,\,10^3,\,10^4$, at two different time instants. Convergence toward the exact solution \eqref{eq:case-study.solution} as the mesh refines can be visually appreciated, although approximating singular measures with densities is a difficult task, which requires very fine and computationally expensive meshes to get accurate results. Therefore, when such a structure of the solution is numerically sought, it is more efficient to exploit the stated equivalence of the original problem with the discrete system of ODEs.

\begin{figure}[t]
\includegraphics[width=\textwidth]{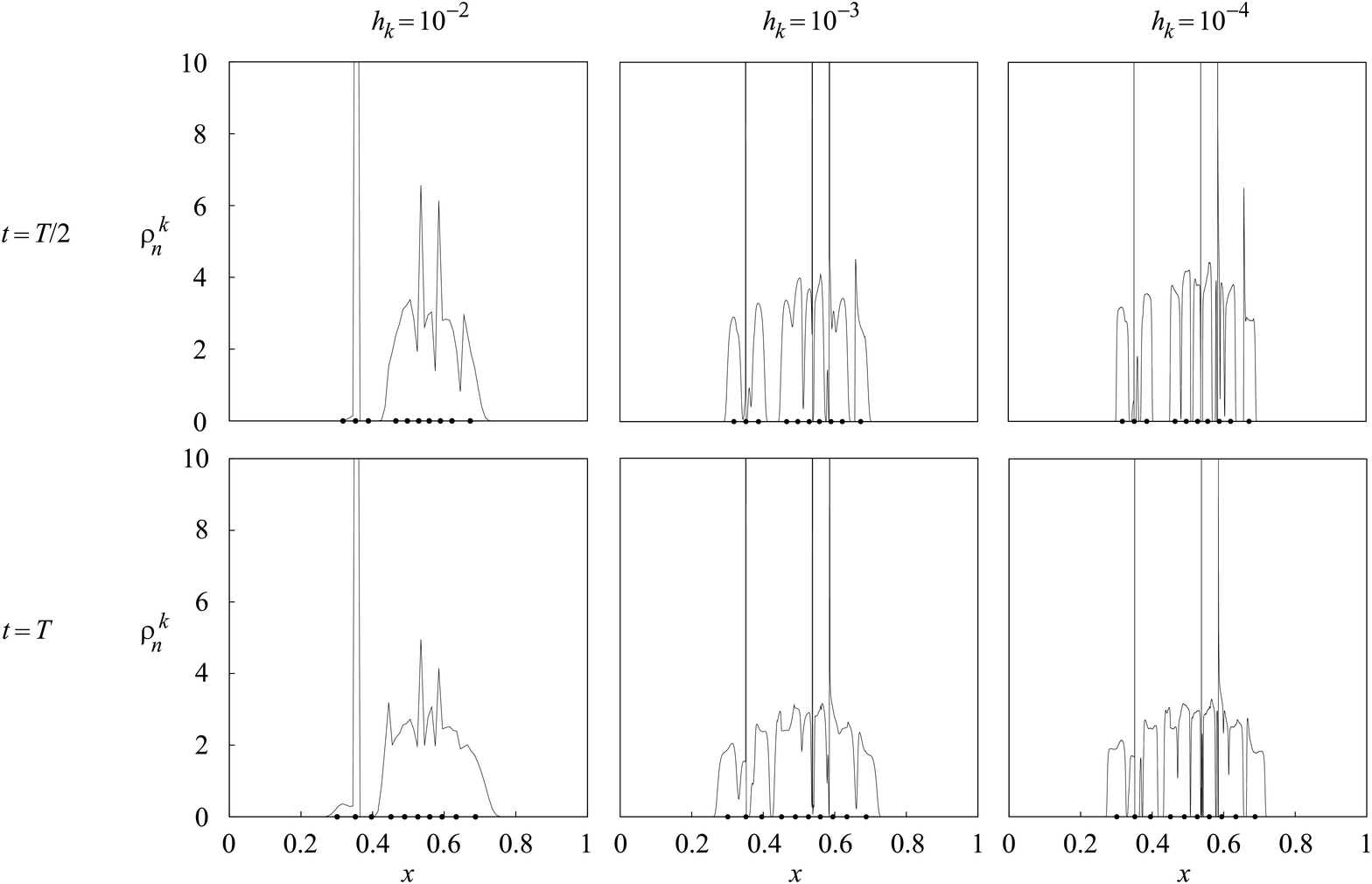}
\caption{Numerical solution to \eq\eqref{eq:cont.eq} with discrete initial datum \eqref{eq:case-study.incond} at two successive times. Bullets are the solution of the ODE system \eqref{eq:case-study.ode} computed by the explicit Euler scheme. The continuous curve is the solution computed by the numerical scheme \eqref{eq:numscheme} on meshes with different level of refinement.}
\label{fig:simulation}
\end{figure}

\section*{Acknowledgments}
A. Tosin was funded by a post-doctoral research scholarship ``Compagnia di San Paolo'' awarded by the National Institute for Advanced Mathematics ``F. Severi'' (INdAM, Italy).

\bibliographystyle{plain}
\bibliography{TaFp-wasserstein}

\end{document}